\documentclass[a4paper,12pt]{article}
\usepackage[utf8]{inputenc}  
\usepackage[T2A]{fontenc}
\usepackage[russian, english]{babel}
\usepackage{graphicx}
\usepackage{listings}
\usepackage[centertags]{amsmath}
\usepackage{amsfonts}
\usepackage{amssymb}
\usepackage{amsthm}
\usepackage{cite}
\usepackage{lastpage}
\usepackage{euscript}
\usepackage{tcolorbox}

\usepackage{hyperref}
\hypersetup{colorlinks=true, linkcolor=blue}

\usepackage{listings}
\usepackage{xcolor} 

\definecolor{codegreen}{rgb}{0,0.6,0}
\definecolor{codegray}{rgb}{0.5,0.5,0.5}
\definecolor{codepurple}{rgb}{0.58,0,0.82}
\definecolor{backcolour}{rgb}{0.95,0.95,0.92}

\numberwithin{equation}{section}
\usepackage{newlfont}
\usepackage{geometry} 
\theoremstyle{plain}
\newtheorem{theorem}{Theorem}[section]
\newtheorem{corollary}{Corollary}[section]
\newtheorem{remark}{Remark}[section]
\newtheorem{proposition}{Proposition}[section]
\newtheorem{lemma}{Lemma}[section]
\theoremstyle{definition}

\usepackage{authblk}

\title{A proof of the Avkhadiev-Wirths conjecture on  Brezis-Marcus constants }

\author[1]{I.\,I.~Gabdulkhalikov \thanks{ilyasmechmat@gmail.com}}
\affil[1]{\small{N.I. Lobachevsky Institute of Mathematics and Mechanics, Kazan Federal University, 35, Kremlievskaia~Str., Kazan, 420008, Russia}}

\author[1,2]{R.\,G.~Nasibullin\thanks{NasibullinRamil@gmail.com}}
\affil[2]{\small{Institute of Software Development and Engineering, Innopolis University, 1, Universitetskaya~Str.,
Innopolis, 420500, Russia}
}
\date{}

\begin{document}

\maketitle

\begin{abstract}
In this paper we deal with geometrical versions of Hardy type inequalities with additional positive terms in convex domains.  The constant $\lambda(\Omega)$ multiplying the additional term depends on the geometry of the multidimensional domain $\Omega$ and the numerical parameters of the problem. The constant (functional) $\lambda(\Omega)$ is called Brezis-Marcus constant. In 2010, F.G. Avkhadiev and K.-J. Wirths  proposed the hypothesis that among all $n$-dimensional domains with given inradius the maximum of the best Brezis–Marcus constant is achieved for the $n$-dimensional ball of radius $\delta_0$.  Using one dimensional Hardy type inequalities we proved the Avkhadiev-Wirths conjecture on  Brezis-Marcus constants in the cases $n=2$ and $n\geq 4$. The sharp constants are solutions of the equation in terms of special functions and fixed  eigenvalues of the Sturm–Liouville differential operators. The corresponding eigenfunctions in the $2$-d case are spheroidal wave functions and for dimensions greater than or equal to $4$ are confluent Heun functions. New properties of the Heun functions are established and their zeros are found. We provide Python code for calculating sharp constants. 



\end{abstract}

\rightline{\textit{We work "near" black holes.}\footnote{We were pleasantly surprised and excited to come  from ``usual Hardy inequalities'' on ``usual balls'' to differential Heun equations related to black holes.}}

\newpage
\renewcommand\contentsname{
\begin{center}
Contents
\end{center}
}
\tableofcontents

\section{Introduction}
This paper is a continuation of a series of works \cite{BFT, AW_Ball, Nas_MS, Ga_Nas} devoted to Hardy type inequalities with additional terms on $n$-dimensional balls. One-dimensional and multidimensional Hardy inequalities with additional terms are well known (see, for example, \cite{Avkh_21, BEL, Maz, Nas_Rew} and references therein). A description of all the results on Hardy-type inequalities with additional terms would take many pages and time.   We simply want to mention the authors who pioneered this area. Research on geometrical versions of Hardy inequalities with remainders has been significantly shaped by the work of F.G.~Avkhadiev  and \mbox{K.-J. Wirths} \cite{AW_Zamm}, H. Brezis and M.~Marcus \cite{BM},  G.~Barbatis, S.~Filippas and A. Tertikas \cite{BFT}, E.B. Davies \cite{Dav, Dav1}, S.~Filippas,  V.~Maz'ya and A.~Tertikas \cite{FMT}, M.~Hoffmann-Ostenhof, T.~Hoffmann-Ostenhof and A.~Laptev \cite{HoHoL}, V.I.~ Levin~\cite{Levin}.   
  
The history of such results goes back to the 1920s when G.H.~Hardy obtained his  one-dimensional inequality \cite{Ha}. Nowadays, Hardy type inequalities arise in different areas of mathematics and mathematical physics, in particular, in the geometric theory of functions and in the theory of isoperimetric inequalities (see, for example, \cite{Avkh_2022, Avkh_98, Nas_2023, AyPeYiAy}).

Let $\Omega$ be a convex domain in the Euclidean space $\mathbb{R}^n$ of dimension $n\geq 2$. In \cite{AW_Zamm}, F.G.~Avkha\-diev and K.-J. Wirths  considered the following functional for $n$-dimensional convex domains 

\begin{equation}\label{intr_f1}
\lambda(\Omega) = \inf_{g\in C_0^1(\Omega)}\left(\int\limits_{\Omega} |\nabla g(x)|^2dx - \frac{1}{4}\int\limits_{\Omega} \frac{|g(x)|^2}{\delta(x)^2}dx \right)\Big/\int\limits_{\Omega} |g(x)|^2dx,
\end{equation}
where the distance function $\delta:\Omega\to (0,\infty)$ from a point $x$ to the boundary $\partial\Omega$ of the  domain $\Omega$ is defined by
$$
\delta(x)= \textrm{dist}(x,\partial\Omega) =\inf_{y\in\partial\Omega}\textrm{dist}(x,y), \quad x\in\Omega,
$$
and $C^1_0 (\Omega)$ is the closure of the family  of smooth functions $f:\Omega \to \mathbb{R}$ with finite Dirichlet integral and supported in $\Omega$.  See \cite{Avkh_Dis_Fu} and  \cite{BEL} where more interesting information on the distance function and its properties may be found.   Here, as usual, by  $\nabla g$ we denote the gradient of   $g$, i.e.
$$
\nabla g(x) = \left(\frac{\partial g(x)}{\partial x_1}, \ldots, \frac{\partial g(x)}{\partial x_n}\right).
$$

The functional $\lambda(\Omega)$ depends on $\Omega$  and is connected with the first eigenvalue $\lambda_1(\Omega)$ of the Laplacian under the Dirichlet boundary condition (see \cite{BEL, Nas_Rew, AW_Zamm}).  For recent studies on estimates for the first eigenvalue of the Laplacian and $p$-Laplacian see, for example, \cite{BrMaz, BobTan, BobTan1}.  One of the most famous related results is Rayleigh-Faber-Krahn isoperimetric inequality \cite{Ban} and it says that for any convex domain   
$$
\lambda_1(\Omega)\geq \frac{\omega_n^{2/n}}{|\Omega|^{2/n}}j_{n/2-1}^2,
$$
where  $\omega_n$ is the volume of the unit ball in $\mathbb{R}^n$, $|\Omega|$ is the volume of the domain $\Omega$ and $j_\nu$ is the first zero of the Bessel function of order $\nu$:
\begin{equation*}
J_\nu(t) = \sum_{k=0}^\infty\frac{(-1)^kt^{2k+\nu}}{2^{2k+\nu}k!\Gamma(k+1+\nu)}, \quad t\in[0,1].
\end{equation*}

The Hersch result \cite{Hersch} implies   that if  a convex domain $\Omega\subset\mathbb{R}^2$ then the following inequality  in terms of the inner radius
$$
\lambda_1(\Omega)\geq \frac{\pi^2}{4}\frac{1}{\delta_0^2(\Omega)}
$$
is valid, where $\delta_0=\delta_0(\Omega)=\sup_{x\in\Omega} \textrm{dist}(x,\partial\Omega)$. Moreover, this sharp result is valid for any $n$-dimensional convex domain \cite{PaSta}. 
In addition, it is known that 
$$
\lambda_1(B_{n}) = \frac{j_{\frac{n}{2}-1}^2}{\delta_0^2(\Omega)},
$$
where $B_{n}$  is an $n$-ball of radius $\delta_0$ (see \cite{Ban}).

The functional $\lambda(\Omega)$ is called Brezis-Marcus constant and it depends on the geometry of the multidimensional domain $\Omega$ and the numerical parameters of the problem.  The problem~(\ref{intr_f1}) might be considered  from various points of view. If we study $\lambda(\Omega)$ for convex domains with finite diameter and  volume, or  with finite inner radius, we will have completely different variational problems \cite{FMT, HoHoL, Nas_Rew, BM, AW_Zamm}.  For example, the first way is to find the best possible $\lambda$ for any convex domain with a finite inner radius.  F.G.~Avkhadiev and K.-J. Wirths \cite{AW_Zamm} established the following  bilateral estimates
\begin{equation*}
  \frac{\lambda_0^2}{\delta_0^2(\Omega)} \leq \lambda(\Omega)\leq \frac{j_{n/2-1}^2-1/4}{\delta_0^2(\Omega)},
\end{equation*}
in terms of the inner radius $\delta_0$ of $\Omega$, where $\lambda_0=0.940\ldots$ is the Lamb constant defined as the first zero in $(0,+\infty)$ of the function 
\begin{equation*}
J_0(z) - 2zJ_1(z).
\end{equation*}

In \cite{AW_Zamm}  it is shown that for any convex domain with finite inner radius $\lambda(\Omega) = \lambda_0^2/\delta_0^2(\Omega)$. Extremal domains for which $\lambda(\Omega) = \lambda_0^2/\delta_0^2(\Omega)$ are linear transformations of the convex domain
$$
(0, 1) \times\mathbb{R}^{n-1}\subset \mathbb{R}^n.
$$

The sharp constants for convex domains with finite diameter and volume are unknown. Notice that the corresponding results by H. Brezis and  M. Marcus from\cite{BM}, and results by M.~ Hoffmann-Ostenhof, T.~Hoffmann-Ostenhof and A. Laptev from \cite{HoHoL}, respectively, imply 
$$
\lambda(\Omega) \geq \frac{1}{4 D^2(\Omega)},\quad \quad  \lambda(\Omega)\geq \frac{1}{4}\frac{K(n)}{|\Omega|^{2/n}},
$$
where $D(\Omega)$ is the  diameter and $|\Omega|$ is the volume of the area of $\Omega$,  $K$ is some positive universal constant. These results have been generalized in various ways in \cite{Nas_2023}.

The second way  is to fix the diameter, the volume or the inner radius of considered convex domains. In \cite{AW_Ball}, F.G. Avkhadiev and K.-J. Wirths  proposed the hypothesis that \textit{among all $n$-dimensional domains with given inradius $\delta_0$  the maximum of the best Brezis–Marcus constants $\lambda(\Omega)$
is presented by $\lambda(B_n)$}, where $B_n $ is an $n$-dimensional ball of radius $\delta_0$:
\begin{equation*}
B_n=\{x\in\mathbb{R}^n:|x-x_0|<\delta_0\}, \quad x_0\in\mathbb{R}^n, \quad \delta_0>0.  
\end{equation*}

The conjecture for $n=1$ is proved by F.G. Avkhadiev and K.-J. Wirths \cite{AW_Ball} and for $n=3$ is confirmed by G. Barbatis, S. Filippas and A. Tertikas \cite{BFT}. It is known that if $n = 1$ and $n = 3$, then, respectively,
\begin{equation*}
\lambda(\Omega) =\frac{\lambda_0^2}{\delta_0(\Omega)}\quad\text{ and  }\quad \lambda(\Omega)= \frac{j_0^2}{\delta_0(\Omega)}.
\end{equation*}

The following inequality 
\begin{equation}\label{intt_f1}
\int\limits_{B_n}|\nabla g(x)|^2dx \geq \frac{1}{4}\int\limits_{B_n}\frac{g^2(x)}{(\rho-|x-x_0|)^2}+\frac{c(n)}{\rho^2}\int\limits_{B_n}g^2(x)dx
\end{equation}
for any function $g:B_n \to \mathbb{R}$   from the closure of the family $C^1_
0 (B_n)$ of smooth functions $g:B_n \to \mathbb{R}$ with
finite Dirichlet integral and supported in $B_n$, is expected to prove this conjecture.   Since $c(n)$ does not depend on linear transformations, it suffices to consider the case $x_0 = 0$ and $\rho = 1$ (see \cite{AW_Ball} for more information).

In this paper, we study the best possible constant $c(n)$ in the Brezis–Marcus inequality~(\ref{intt_f1}). If $n\in \mathbb{N}\backslash\{1,3\}$, then the conjecture remains open and only two-sided estimates are obtained (see \cite{AW_Ball, BFT}). Namely,
\begin{equation*}
2\leq c(2)\leq j^2_{0}-\frac{1}{4},
\end{equation*}
\begin{equation*}
j_0^2+\frac{(n-1)(n-3)}{4} \leq c(n) \leq  j^2_{\frac{n}{2}-1}-\frac{1}{4},\quad \text{if} \quad n\geq 4.
\end{equation*}
We should add  the lower estimates for $c(n)$, $n \geq 4$, were for the first time  proved  in \cite{BFT} in 2003. They later  were reproved in 2012 in \cite{AW_Ball} by different types of arguments. In addition, \cite{AW_Ball} contains some  supplementary  material. 

To prove these estimates in the papers \cite{BFT, AW_Ball, Nas_MS, Ga_Nas} different and interesting methods and approaches are used. G. Barbatis, S. Filippas and A. Tertikas  combined a vector field approach  along with ideas of \cite{HoHoL}. The “mean distance” method of Davies ( \cite{Dav, Dav1}) played an essential role. In  \cite{AW_Ball}, F.G. Avkhadiev and K.-J. Wirths used one-dimensional inequalities. Namely, they for any continuously differentiable function $f:[0,1]\to\mathbb{R}$ such that $f(0)=0$ proved that the following inequalities  hold
\begin{equation*}
\int\limits_0^{1}f'^2(t)(1-t)dt\geq \frac{1}{4}\int\limits_0^{1}\frac{f^2(t)}{t^2}(1-t)dt+2\int\limits_0^{1}f^2(t)(1-t)dt,
\end{equation*}
\begin{equation*}
\int\limits_0^{1}f'^2(t)(1-t)^2dt\geq\frac{1}{4}\int\limits_0^{1}\frac{f^2(t)}{t^2}(1-t)^2dt+ j_0^2\int\limits_0^{1}f^2(t)(1-t)^2dt
\end{equation*}
and if $n\geq 4$, then 
\begin{multline*}
\int\limits_0^{1}f'^2(t)(1-t)^{n-1}dt\geq\frac{1}{4}\int\limits_0^{1}\frac{f^2(t)}{t^2}(1-t)^{n-1} dt
+\left(j_0^2+\frac{(n-1)(n-3)}{4}\right)\int\limits_0^{1}f^2(t)(1-t)^{n-1}dt.
\end{multline*}

Using one-dimensional inequalities, the lower estimates of $c(n)$, when $n=2,4, 5, \ldots, 10$, were improved  by R.G. Nasibullin in \cite{Nas_MS} (see, in addition, \cite{Ga_Nas}).  He succeeded in choosing (guessing) the form of the special functions
and numerical parameters involved. In proofs he estimated weight functions in terms of piece-wise convex functions. That is, continuous functions that are convex on each sub-interval of the corresponding mesh. As a result, the constants are obtained as minima of convex functions. R.G. Nasibullin showed that, for example,
\begin{equation*}
c(2)\geq  2.44382 \quad \text{and} \quad c(4) \geq  6.92837.
\end{equation*}

Also we refer to paper \cite{Ba24} by G.~Barbatis, in which the author  concentrates on other results where the best Hardy constant is either computed exactly or estimated from below.

In this paper, we give a proof of the Avkhadiev-Wirths conjecture on  Brezis-Marcus constants in the case $n\in\mathbb{N}\setminus\{1,3\}$. To achieve our goal, we establish new sharp one dimensional inequities of Hardy type. The sharp constants are solutions of the equation in terms of special functions and fixed eigenvalues of  Sturm–Liouville differential operators. The corresponding eigenfunctions in the $2$-d case are spheroidal wave functions \cite{Fl} and for dimensions greater than or equal to $4$ are confluent Heun functions \cite{Ronveaux}. In  Table \ref{tab:my_label} below\footnote{The cases $n=1$ and $n=3$ is from \cite{BFT, AW_Ball}. We included them solely to provide complete context.}, we compare the results from the mentioned papers.

\begin{table}[h]
    \centering
    \begin{tabular}{ccccc}
     \hline
      $n$   & Results from \cite{Nas_MS} & Results from \cite{Ga_Nas}& Sharp constants & Upper bounds   \\  \hline
      $1$    &     --     &    --   & 0.88504925399\ldots &$2.2174011$   \\
      $2$    & $2.44382$  & 2.952   & 2.95271973692\ldots &$5.53319$  \\
      $3$    &    --      &  --     & 5.78318596294\ldots &$9.61960$   \\
      $4$    & $6.92837$  &  8      & 9.31512370489\ldots &$14.432$   \\
      $5$    & $9.39764$  &  9.5513 & 13.5148747256\ldots &$19.9407$  \\
      $6$    & $11.9419$  &  13.2274& 18.36050676851\ldots &$26.1246$  \\
      $7$    & $14.5342$  &  17.3126& 23.83631772301\ldots &$32.9675$ \\
      $8$    & $17.1324$  &  21.8592& 29.93037875461\ldots&$40.4565$  \\
      $9$    & $19.7308$  &  26.8856& 36.63324667872\ldots&$48.5812$  \\
      $10$   & $22.3292$  &  32.4   & 43.93721660080\ldots&$57.3329$  \\
       \hline
    \end{tabular}
    \caption{Comparison of known  lower estimates of $c(n)$ with sharp constants}
    \label{tab:my_label}
\end{table}

Our proof of the Avkhadiev-Wirths conjecture contains several main steps. First, we will show that a proof of multidimensional inequality of Hardy type is based on one-dimensional weighed inequalities. In the process of proving inequalities, differential equations of seconds order arise. Solutions of these differential equations must satisfy proper conditions.

Second, we use spheroidal wave and Heun differential equations of second order and their properties to prove the one-dimensional inequalities. Our technique fixes the corresponding eigenvalues of the Sturm-Liouville equations. 

Third, we establish sharpness of the constants. As usual for Hardy type inequalities, there are no extremal functions at which equality is attained.  We construct minimizing sequences. These sequences depend on  spheroidal wave functions or on confluent  Heun functions.  

That is why this paper is organized as follows. In section \ref{sec1}, we provide  brief information on spheroidal and Heun functions and their properties. There, we consider corresponding differential equations as eigenvalues and eigenfunction problem.  Section \ref{sec2} is devoted to the spheroidal and Heun functions in the particular cases. Different forms of the equations are considered. Some preliminary statements are also established. In the last section, we deal with one dimensional inequalities of Hardy type. We obtain  sharp inequalities when  the dimension $n$ is equal to $2$ and $n\geq 4$.  These sharp one-dimension inequalities prove the Avkhadiev-Wirths conjecture on Brezis-Marcus constants. We provide Python code that  computes the value of sharp constants with any given precision (see Appendices A and B).  Our bilateral estimates of the sharp constants from \cite{Ga_Nas} were instrumental in applying numerical methods to compute the roots of equations expressed in terms of Heun functions.  In Python cod we use high precision ($40$~digits) to calculate constants, and Table ~\ref{tab:my_label8} from Appendix C shows the constants calculated with this parameter. For larger $n$, we recommend increasing this parameter to improve the stability of our algorithm. The paper contains supporting illustrations and graphics.

The main statements of the paper and the tools used to prove them lay in the intersection of different areas of mathematics and mathematical physics.  We believe that our results might be useful in future research, especially in theory of Hardy type inequalities, isoperimetric inequalities, and in some areas of physics.

\section{On spheroidal and Heun functions}\label{sec1}

\subsection{The angular spheroidal functions and their properties}\label{sec21}
In this section we will consider necessary information on angular spheroidal functions and their properties. For more details, we refer to the following books \cite{Fl, KPS_Book} and the article \cite{Rh}.

Consider  the second order Sturm–Liouville differential operator $L$ defined by
$$
Ly: = \frac{d}{dz}\left((1-z^2)\frac{dy(z)}{dz}\right) +\left(\gamma^2(1-z^2)-\frac{m^2}{1-z^2}\right)y(z),
$$
where $\gamma$ and $m$ are numbers. Spheroidal wave functions are a class of special functions that satisfy the differential equation 
$$
Ly(z) = \lambda_\nu(\gamma, m) y(z)
$$
with $z\in(-1,1)$ and  the following boundary conditions
$$
|y(-1)|<\infty\quad\text{and}\quad |y(1)|<\infty.
$$
The eigenvalue problem can be written as
\begin{equation}\label{StLiEq}
(1-z^2)y''(z)-2z y'(z) +\left(-\lambda_\nu(\gamma, m)+\gamma^2(1-z^2)-\frac{m^2}{1-z^2}\right)y(z)=0.
\end{equation}
Equation (\ref{StLiEq}) has two regular singularities at $-1$ and $1$, and an irregular singularity at $\infty$.

Denote the normalizable solution of equation (\ref{StLiEq}) by $S_{m,\nu}^1(\gamma,z)$. We will choose the numbering of the functions $S_{m,\nu}^1(\gamma,z)$ in such a way that they have $\nu-m$ zeros on the interval $(-1, 1)$, so that $\nu- m\geq 0$. The such choice of numbering gives an increased sequence of $\lambda_\nu(\gamma, m)$ and, in addition, $\lambda_\nu(\gamma, m)\to \infty$ as $\nu \to 0$. The parameter $\nu$ enumerates the spheroidal eigenvalues in such a manner that 
$
\lim_{\gamma\to 0} \lambda_\nu(\gamma, m) = \nu(\nu+1)
\quad
 \text{and} 
\quad
\lim_{z\to 0}S_{m,\nu}^1(\gamma,z/\gamma)= j_\nu(z),
$
where $j_\nu(z)$ is the spherical Bessel function of the first kind. 

In the sequel we need the following properties of spheroidal wave functions $S_{m,\nu}^1(\gamma,z)$ (see, for example, \cite{KPS_Book, Rh}):

\textit{Property 1.} Sturm–Liouville equation (\ref{StLiEq}) possess an infinite set of discrete eigenvalues  $\lambda_\nu(\gamma, m)$ and corresponding, non-zero orthogonal eigenfunctions $S_{m,\nu}^1(\gamma,z)$.

\textit{Property 2.} The equation $S_{m,\nu}^1(\gamma,z)=0$  has  $\nu-m$ roots on the interval $(-1,1)$.  Moreover, $$(1-z^2)^{-\frac{m}{2}}S_{m,\nu}^1(\gamma,z)$$ is an entire function. If $m=0$ then  $S_{0,\nu}^1(\gamma,z)$ has no roots in $[-1,1]$. See for example Fig. \ref{Fig1}

\textit{Property 3.} The eigenvalue $\lambda_\nu(\gamma, m)$ is an increasing function with respect to $\nu$. It means
$$
\lambda_0(\gamma, m)<\lambda_1(\gamma, m)<\lambda_2(\gamma, m)<\ldots 
$$
and 
$$
\min_\nu \lambda_\nu(\gamma, m) = \lambda_0(\gamma, m).
$$
Moreover,
$$
\lambda_\nu(\gamma, m)\to \infty \quad \text{as} \quad \nu \to \infty.
$$

\begin{figure}[h]
\begin{center}
\includegraphics[scale=0.6]{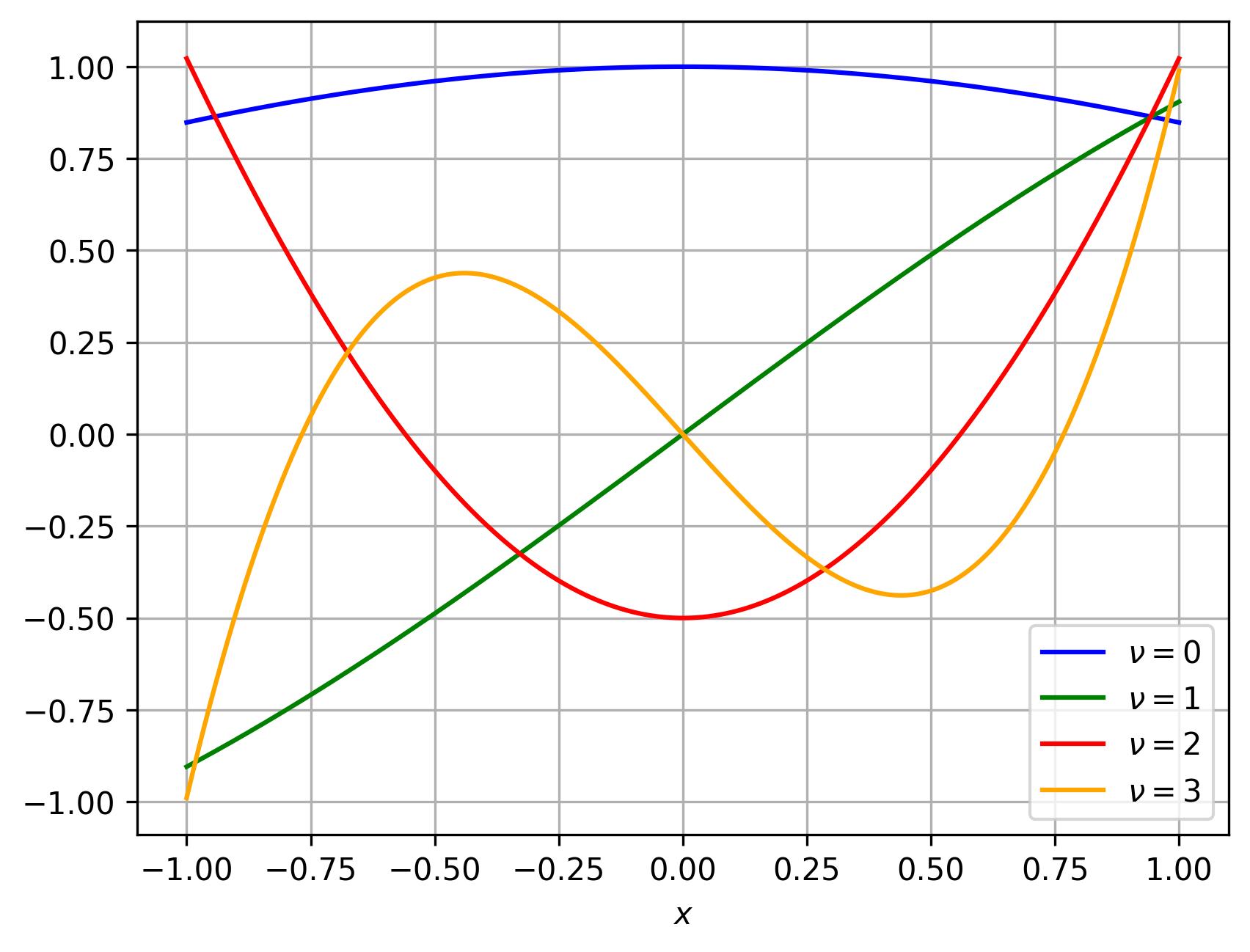}
\caption{Prolate spheroidal function $S_{m,\nu}^1(\gamma,x)$ with fixed $\gamma=1$ and $m=0$}
\label{Fig1}
\end{center}
\end{figure}

\textit{Property 4.} If $\nu$ and $m$ are fixed, then $\lambda_\nu(\gamma, m)$  is a decreasing function with respect to $\gamma$. In addition, it is known that

$$
\lambda_\nu(\gamma, m)=-\gamma^2+\gamma \left(2(\nu-m)+1\right)+m^2+O(1/\gamma).
$$ 
Hence  $\lambda_\nu(\gamma, m)\to -\infty$ as $\gamma\to \infty.$

\textit{Property 5.} If $\nu$ and $\gamma$ are fixed, then $\lambda_\nu(\gamma, m)$  is a increasing function with respect to $m$.
 
\textit{Property 6.} If $\nu$ and $m$ are integer numbers, then 
$$
S_{m,\nu}^1(\gamma,z) = (-1)^{\nu-m}S_{m,\nu}^1(\gamma,-z).
$$

\textit{Property 7.}  To compute $\lambda_\nu(\gamma,m)$  with any given precision the following  continued fraction
$$
\lambda+\gamma^2-m(m+1)+\frac{1\cdot2\cdot\gamma^2}{\lambda+\gamma^2-(m+2)(m+3)+\frac{3\cdot4\cdot\gamma^2}{\lambda+\gamma^2-(m+4)(m+5)+...}}
$$
may be used. See \cite{KPS_Book} for more details. 

\subsection{Confluent Heun functions}\label{sec22}
In this section we will consider confluent Heun functions. See  \cite{Ronveaux} and references therein for more information. Note, particular and limiting cases of the Heun equation are the Lam\'{e}  equation, the equations of spheroidal and Coulomb spheroidal functions, the hypergeometric equation, the Legendre equation, the Bessel equation, and others. To additional information on applications of  Heun functions we refer to \cite{Hort, LaSla, MiPaMa}.

Consider  the second order Sturm–Liouville differential operator $L$ defined by
$$
Ly: =  y''(z)+\left( \frac{\gamma}{z} +\frac{\delta}{z-1}  +\varepsilon \right)y'(z)+ \frac{\alpha }{ (z-1)} y(z), \quad z\in (0,1).
$$
The confluent Heun function $H_C$ solves the following  spectral problem
$$
Ly(z) = \frac{q}{z(z-1)}y(z), \quad z\in (0,1) .
$$
Consequently, $H_C$ satisfies the confluent Heun differential equation 
\begin{equation}\label{HeunC_eq}
    y''(z)+\left(\frac{\gamma}{z} +\frac{\delta}{z-1} +\varepsilon\right)y'(z)+ \frac{\alpha z-q}{z(z-1)} y(z)=0, \quad z\in (0,1).
\end{equation}
Boundary conditions are specified at singular points of the equation. Equation (\ref{HeunC_eq}) has three singular points: two regular ones $z = 0$ and $z = 1$, and one irregular one $z = \infty$. The function $H_C$ is the regular solution of the confluent Heun equation that satisfies the condition $H_C(q,\alpha,\gamma,\delta,\varepsilon,0)=1$ (see Fig. \ref{Fig2}).  In addition, 
$$
H_C(t) \approx 1 + \frac{q}{\gamma} t + O(t^2).
$$

\begin{figure}[h]
\begin{center}
\includegraphics[scale=1]{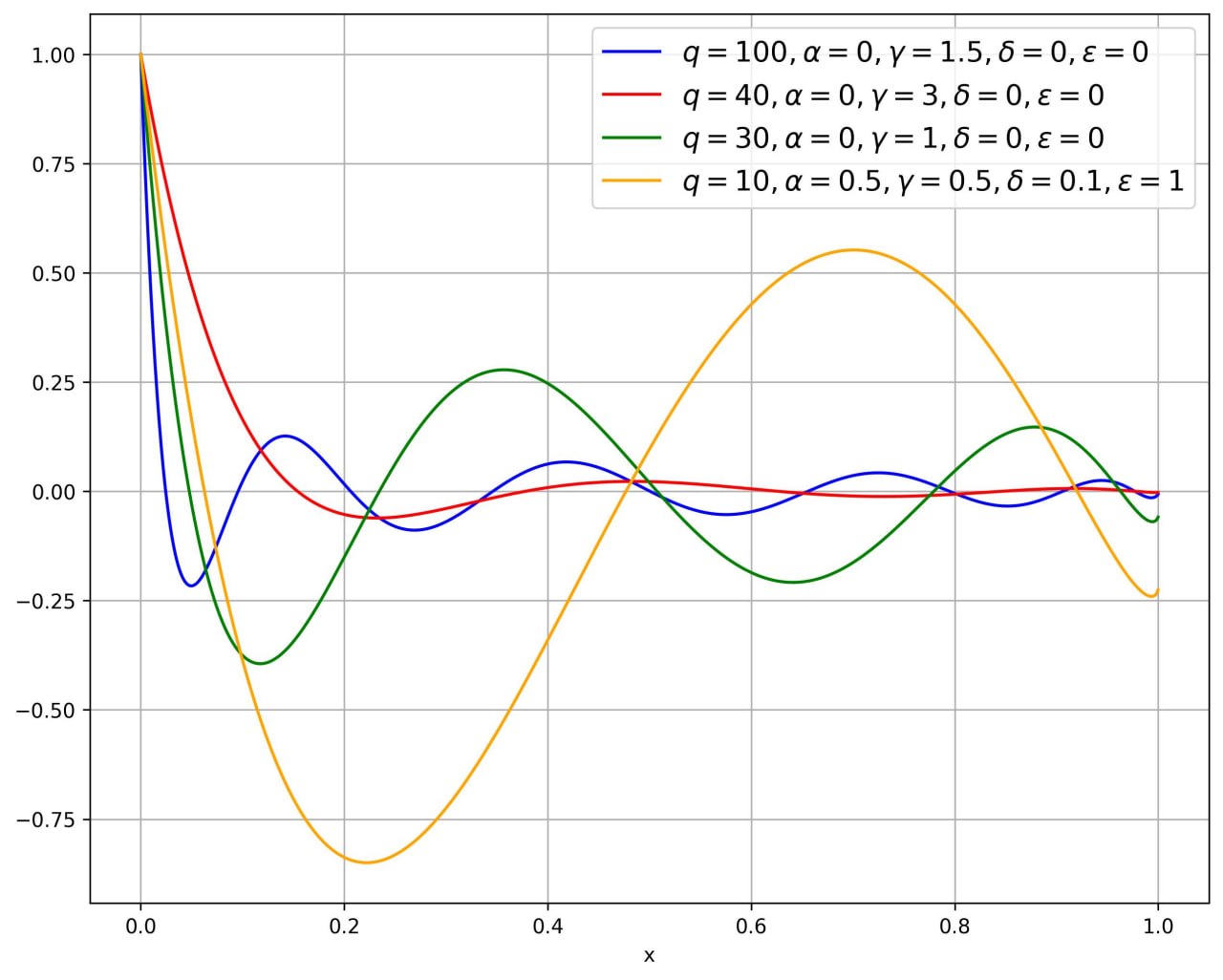}
\caption{$H_C$ function with different parameters}
\end{center}
\label{Fig2}
\end{figure}

In the neighbourhood of the point $1$, the confluent Heun function decomposes into a linear combination of two independent local solutions:
$$
H_C(z) = C_1 \cdot y_1(z) + C_2 \cdot y_2(z),
$$
where  
$$
y_1(z) = 1 + c_1(z-1) + c_2(z-1)^2 + \dots
$$ is regular function and  the function
$$
y_2(z) = (z-1)^{1-\delta} \left(1 + d_1(z-1) + \dots \right)
$$ is  irregular at $1$. In general, the constants $C_1 =C_1(q,\alpha,\gamma,\delta,\varepsilon)$ and  $C_2= C_2(q,\alpha,\gamma,\delta,\varepsilon)$ depend on all parameters of the equation. Hence
$$
H_C(q,\alpha,\gamma,\delta,\varepsilon,t)\sim \frac{1}{(1-t)^{\delta-1}}\quad \text{as} \quad t\to 1, \quad\text{if}\quad  \delta-1>0. 
$$
Moreover,  confluent Heun functions satisfy the equality
$$H_C(q, \alpha, \gamma, \delta, \epsilon, t) = (1-t)^{1-\delta} H_C(\tilde{q}, \tilde{\alpha}, \tilde{\gamma}, \tilde{\delta}, \tilde{\epsilon}, t),
$$
where 
$$
\tilde{q} = q - \gamma(1-\delta),
\quad
\tilde{\alpha} = \alpha + \epsilon(1-\delta),
$$
$$
\tilde{\gamma} = \gamma,
\quad
\tilde{\delta} = 2 - \delta,
\quad
\tilde{\epsilon} = \epsilon.
$$

\section{Preliminary statements }\label{sec2}

When proving one dimensional inequalities we are faced with differential equations. The differential equations are not in standard and known forms. In this section,  we obtain the solutions of the differential equations and establish some of their properties. 

\subsection{The angular spheroidal functions}

\begin{lemma}\label{le1}
The solution of the following second order differential equation 
$$
(1-t)y''(t)-y'(t)+y(t)(1-t)\left(\frac{1}{4t^2}+A^2\right)=0, \quad t\in (0,1),
$$
where $A$ is a constant such that 
$$
y(0) = 0 \quad\text{and}\quad  0\neq y(1) <\infty,
$$
 is 
$$
y(t) = \sqrt{t} S_{0,\nu}^1\left(\frac{A}{2}, 2t-1\right).
$$
\end{lemma}
\begin{proof}
We will search for a solution of the equation in the form
$
y(t) = \sqrt{t}w\left(t\right).
$
Consequently, $w(t)$  satisfies the equation 
$$
(1 - t) t w''(t)+ (1 - 2 t) w'(t) + \left(A^2 (1 -t) t-\frac{1}{2} \right) w(t)=0.
$$
If for any  $t\in [0,1]$ we put $x = 2t-1$ then we will have $x\in[-1,1]$ and the last equation takes the form
$$
(1-x^2)w''(x)-2x w'(x) +\left(-\frac{1}{2}+\frac{A^2}{4}(1-x^2)\right)w (x)=0.
$$
It is easy to see that  this equation coincides with (\ref{StLiEq}) for $m=0$ and $\lambda_\nu(0,A/2)=-1/2$. Therefore, $w(x)= S_{0,\nu}^1(A/2, x)$ and  the initial solution 
$$
y(t) = \sqrt{t} S_{0,\nu}^1\left(\frac{A}{2}, 2t-1\right).
$$ 
Using Property 2, we can show that the found solution satisfies all boundary conditions. Indeed
$$
y(0)= 0, \quad \quad  y(1) \neq 0, \quad\text{and}\quad y(1)<\infty.
$$
This completes the proof of Lemma \ref{le1}.
\end{proof}

\begin{lemma}\label{kappa_comp} The  following equation has an unique solution 
$$
\lambda_0(\kappa/2,0)=-\frac{1}{2}.
$$
In addition, $\kappa=1.718347\ldots$.
\end{lemma}
\begin{proof}
Since $\lambda_\nu(0,m)=\nu(\nu+1)=\max_{\gamma}\lambda_\nu(\gamma,m)$ and  $\lambda_\nu(\gamma,m)$ is increasing with respect to $\gamma$ and is unbounded below (Property 4) we obtain that the equation 
$$
\lambda_0(\kappa/2,0)=-\frac{1}{2}
$$ 
has a unique root (see Fig. \ref{kappa}). 

\begin{figure}[h]
\begin{center}
\includegraphics[scale=0.6]{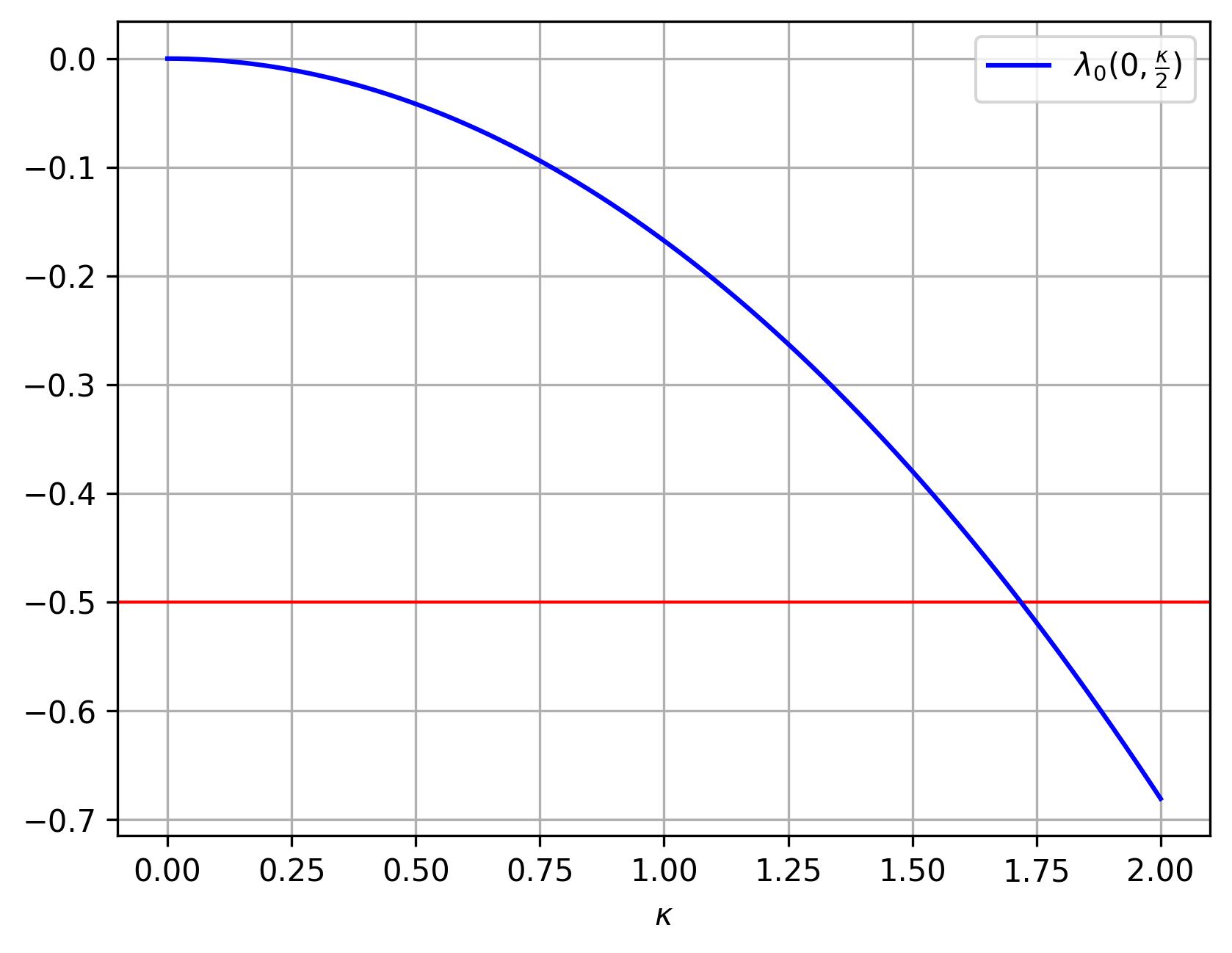}
\caption{The intersection of $\lambda_0(0,\kappa/2)$ and the line $-1/2$}
\label{kappa}
\end{center}
\end{figure}

Using Property 7, we can find $\kappa$ for $\lambda=-\frac{1}{2}$ with any given precision. In our case $m=0$ and $\gamma^2=\kappa^2/4$. We fix any big depth of the continued fraction and solve the corresponding equation by standard  numerical methods. In addition, we provide Python code (see Appendix A) to get the numerical solutions of the equation.
\end{proof}

\subsection{Heun functions}\label{sec3.2}

The proof of one-dimensional inequalities in the case $n\geq 4$ is based on Heun's second order differential equations. In the sequel we need to use  different forms of the equations and properties of their solutions.

The following lemma holds.
\begin{lemma}\label{le2}
The solution of the second order differential equation 
\begin{equation}\label{diff_eq_H}
    y''(t)-\frac{n-1}{1-t}y'(t)+y(t)\left(\frac{1}{4t^2}+A^2\right)=0, \quad t\in (0,1),
\end{equation}
s.t.
$$
y(0) = 0 \quad \text{and} \quad y(1)<\infty,
$$
is the following  real-valued function 
$$
y(t)=\frac{\sqrt{t}e^{iAt}}{(1-t)^{n-2}}H_C\left(\frac{n-3}{2} + i A ,iA(4-n),1,3-n,2iA,t\right),
$$
where $A$ is a constant and by $H_C$ the confluent  Heun function is denoted.
\end{lemma}
\begin{proof}
The proof of the lemma consists of two steps. In the first step, the solution is obtained, and in the second step, we prove that the solution is a real-valued function.

If we  substitute
$$
y(t)=\frac{w(t)}{(1-t)^{n-2}}
$$ 
into the original differential equation (\ref{diff_eq_H}), then we get
\begin{equation}\label{the_first_form}
    w''(t)+\frac{n-3}{1-t}w'(t)+\left(\frac{1}{4t^2}+A^2\right)w(t)=0.
\end{equation}
Solutions of the last differential equation  we will search in the form $w(t)=\sqrt{t}h(t)$. We have

\begin{equation}\label{the_sec_form}
t(1-t)h''(t)+(1+(n-4)t)h'(t)+\left(A^2t(1-t)+\frac{n-3}{2}\right)h(t)=0.
\end{equation}
Moreover, the last equation we can rewrite as
\begin{equation}\label{the_sec_form1}
h''(t)+\frac{1+(n-4)t}{t(1-t)}h'(t)+\left(A^2+\frac{n-3}{2t(1-t)}\right)h(t)=0.
\end{equation}
The replacement $h(t)=e^{i\lambda t}u(t)$ gives
$$
u''(t)+\left(\frac{1}{t}+\frac{n-3}{1-t}+2iA\right)u'(t)+\frac{t(n-4)iA-(\frac{3-n}{2}-iA)}{t(1-t)}u(t)=0.
$$
Consequently, we have the confluent non-symmetrical Heun equation
$$
u''(t)+\left(\frac{\gamma}{t} +\frac{\delta}{t-1} +\varepsilon\right)u'(t)+ \frac{\alpha t-q}{t(t-1)} u(t)=0,
$$
where $q=iA+\frac{n-3}{2}$, $\alpha=(4-n)i\lambda$, $\gamma=1$, $\delta=(3-n)$ and $ \varepsilon=2iA$ (see, for example, \cite{Ronveaux}).

Therefore
$$
y(t)= \frac{\sqrt{t}e^{Ait}}{(1-t)^{n-2}} H_C\left(iA+\frac{n-3}{2},iA (4-n),1,3-n,2iA,t\right).
$$
Note we can substitute  
$$
y(t)= \frac{\sqrt{t}e^{Ait}}{(1-t)^{n-2}}u(t)
$$
into the initial equation from the lemma and get the solution without intermediate results. We will use the intermediate results further.

Next we will show that the solution $y(t)$ is a real valued function.
It is enough to show that 
$$
h(t)=e^{Ait} H_C\left(iA+\frac{n-3}{2},iA (4-n),1,3-n,2iA,t\right)
$$ 
is real-valued.

It is known that Heun function $h(t)$ might be computed as a power series expansion around the origin, a regular singular point: 
$$
h(t)=\sum_{k=0}^{\infty}a_kt^{k},
$$
where $a_k=h^{(k)}(0)/k!$. The coefficients $a_k$ are real numbers. The proof is by induction on $k$. Indeed, since $h$ is the solution of the differential equation
$$
t(1-t)h''(t)+(1+(n-4)t)h'(t)+\left(A^2t(1-t)+\frac{n-3}{2}\right)h(t)=0
$$
and $h(0)=1$, we get 
$$
a_1 = h'(0)=\frac{3-n}{2}.
$$
Suppose now that for all $l=0,1,\ldots, m$ the coefficient $a_l= h^{l}(0)$ is a real number. Let us show that $h^{k+1}(0)$ is also real. Differentiating the above equation $l$-times, we obtain
$$
\sum_{k=0}^{l}\binom{l}{k}(t(1-t))^{(l-k)}h(t)^{(k+2)}+\sum_{k=0}^{l}\binom{l}{k}(1+(n-4)t)^{(l-k)}h(t)^{(k+1)}+$$
$$+\sum_{k=0}^{l}\binom{l}{k}\left(A^2t(1-t)+\frac{n-3}{2}\right)^{(l-k)}h(t)^{(k)}=0.
$$
Consequently, 
$$
(l+1)h^{(l+1)}(0)+\left(-l(l-1)+l(n-4)+
\frac{n-3}{2}\right)h^{(l)}(0)+
$$
$$
lA^2h^{(l-1)}(0)-l(l-1)A^2h^{(l-1)}(0)=0.
$$
Therefore,  $h^{(l+1)}(0)$ is real. This completes the proof of Lemma \ref{le2}.
\end{proof}

Also we need the following lemma,  where some properties of $h(t,A)$ are considered. These properties from the Lemma  help us to understand the behaviour of the graph of the function $h(t,A)$ whether this function increases or decreases,  whether the function has minimum or maximum and so on. 

\begin{lemma}\label{cor_prop}
For any positive $A$ the function $h(t)= h(t,A)$ has the following properties
\begin{enumerate}
    \item{ \begin{equation}\label{der_h_at_1}
h'(0,A) = -\frac{n-3}{2} \quad\text{and}\quad h'(1,A) = -\frac{1}{2}h(1,A).
\end{equation}} 
    \item{
If $n\geq 5$ then
\begin{equation}\label{sec_der_h_at_1}
h''(1,A) = \left(\frac{A^2}{n-4}+\frac{3n-11}{4(n-4)}\right)h(1,A),
\end{equation}
\begin{equation*}
h''(1,A) = -2\left(\frac{A^2}{n-4}+\frac{3n-11}{4(n-4)}\right)h'(1,A).
\end{equation*}
If $n=4$ then $h''(t,A)$ tends to $\infty$ as $t\to 1$.
} 
\item{
$$
w'(1,A) = 0 \quad\text{and}\quad w''(1) = \frac{1}{n-2}(1+A^2)w(1).
$$
} 
\item{ 
If $t_0\in(0,1)$ and $t_0$ is an extremum point then
$$
w''(t_0)=-\left(\frac{1}{4t_0^2}+A^2\right)w(t_0).
$$
} 
\end{enumerate} 
\end{lemma}

\begin{proof}
If in (\ref{the_first_form}) $t\to 1$  then we obtain
$$
(n-3)h'(1)+\left(\frac{n-3}{2}\right)h(1)=0.
$$
Hence we have property 1.

Consider the following limit
$$
L= \lim_{t\to 1}h''(t)+A^2h(t)+\frac{(1+(n-4)t)h'(t)+\frac{n-3}{2}h(t)}{t(1-t)}.
$$
Using L'Hpital's rule we get
$$
L= \lim_{t\to 1}h''(t)+A^2h(t)+\frac{(n-4)h'(t)+(1+(n-4)t)h''(t)+\frac{n-3}{2}h'(t)}{1-2t} =
$$
$$
h''(1)+A^2h(1)-(n-4)h'(1)-(n-3)h''(1)-\frac{n-3}{2}h'(1)=
$$
$$
(4-n)h''(1)+A^2h(1)-\frac{3n-11}{2}h'(1) .
$$
Hence applying  (\ref{the_sec_form1}) and  using the equality $h'(1,A) = -\frac{1}{2}h(1,A)$, we obtain
$$
(n-4)h''(1) = \left(A^2+\frac{3n-11}{4}\right)h(1) .
$$
Thus we have property 2.

Equation (\ref{the_first_form})  implies property 3.
\end{proof}

\subsection{Existence of zeros}\label{sec3.3}

The following singular Sturm comparison theorem holds (see, for example \cite{Swan}, p. 5).  

\begin{theorem}[St]
Let  $x(t)$ be a non-trivial solution of 
$$
(p(t)x'(t))'+q(t)x(t) = 0
$$
on $a<t<b$. Let  $p$, $q$, $P$, $Q$ be continuous on $a<t<b$ and let $0\leq p\leq P$, $q\leq Q$.   If $x(t_1)=x(t_2)=0$ for $a<t_1<t_2<b$, then any solution $X(t)$
 of 
$$
(P(t)X'(t))'+Q(t)X(t) = 0
$$
 which is not a constant multiple of $x(t)$
 in case $p=P$
 and 
 $q=Q$
 has at least one zero in the open interval $(t_1,t_2)$. 
\end{theorem}
We will apply Sturm comparison theorem to find or estimate number of roots of  the second order linear differential equation 
\begin{equation}\label{sec33_f1}
t(1-t)y''(t)+(1+(n-4)t)y'(t)+\left(A t(1-t)+\frac{n-3}{2}\right)y(t)=0.
\end{equation}

It is known that the substitution
$$
y(t)=u(t)\cdot exp\left(-\frac{1}{2}\int p(t)dt\right)
$$
reduces the differential equation  
$$y''+py'+q=0,$$ 
to the form
$$
u''+Q u=0.
$$
In our case we have
$$
u''+\left(\frac{1}{4t^2}+A^2+\frac{(n-1)(3-n)}{4(1-t)^2}\right)u=0.
$$
We are going to compare the second order linear differential equation  with the following equation 
$$
u''(t)+\frac{k^2\pi^2}{(1-\varepsilon_0-\varepsilon_1)^2}u(t)=0, \quad t\in(\varepsilon_0,1-\varepsilon_1),
$$
subject to 
$$
u(\varepsilon_0) = u(1-\varepsilon_1)=0,
$$
where real number $\varepsilon\in(0,1)$ and $k\in \mathbb{N}$. The solution of the equation is
$$
u(t)= \sin\left(\frac{k\pi}{1-\varepsilon_0-\varepsilon_1}(t-\varepsilon_0)\right).
$$
If $A$ such that 
$$
\frac{1}{4t^2}+A^2-\frac{(n-1)(n-3)}{4(1-t)^2}> \frac{k^2\pi^2}{(1-\varepsilon_0-\varepsilon_1)^2}
$$
in the bounded interval $(\varepsilon_0,1-\varepsilon_1)$ then by Sturm theorem the solution of the equation (\ref{sec33_f1}) has at least one zero in the open interval $(\varepsilon_0,1-\varepsilon_1)$. 
Since 
$$
\frac{1}{4t^2}+A^2-\frac{(n-1)(n-3)}{4(1-t)^2}> \frac{1}{4(1-\varepsilon_1)^2}+A^2+\frac{(n-1)(3-n)}{4\varepsilon_1^2},
$$
we get the condition 
$$
A^2>\frac{k^2\pi^2}{(1-\varepsilon_0-\varepsilon_1)^2}-\frac{1}{4(1-\varepsilon_1)^2}+\frac{(n-1)(n-3)}{4\varepsilon_1^2},
$$
which is guaranteed existence of at least one zero in the open interval $(\varepsilon_0,1-\varepsilon_1)$. 

The following statement holds. 

\begin{lemma}\label{le3}
There exists $A_0(n)$ such that if $A\in [0,A_0(n)]$ then the function 
$$
h(t,A) = e^{iAt}H_C\left(\frac{n-3}{2} + i A ,iA(4-n),1,3-n,2iA,t\right)
$$
is positive on $(0,1)$ and 
$$
h(1,A_0(n)) = e^{iA_0(n)}H_C\left(\frac{n-3}{2} + i A_0(n) ,iA_0(n)(4-n),1,3-n,2iA_0(n),1\right) = 0.
$$
\end{lemma}
\begin{proof}
The function $h(t,A)$ satisfies differential equation (\ref{the_sec_form1}). Using the boundary condition we get
$$
h(0,A) = H_C\left(\frac{n-3}{2} + i A ,iA(4-n),1,3-n,2iA,0\right)=1.
$$
Hence Property 1  from Lemma \ref{cor_prop} gives that the function  $h(t,A)$ is  decreasing (see Fig. \ref{sketch_h(n)}). 

\begin{figure}[h]
\begin{minipage}[h]{0.5\linewidth}
\center{\includegraphics[width=1\linewidth]{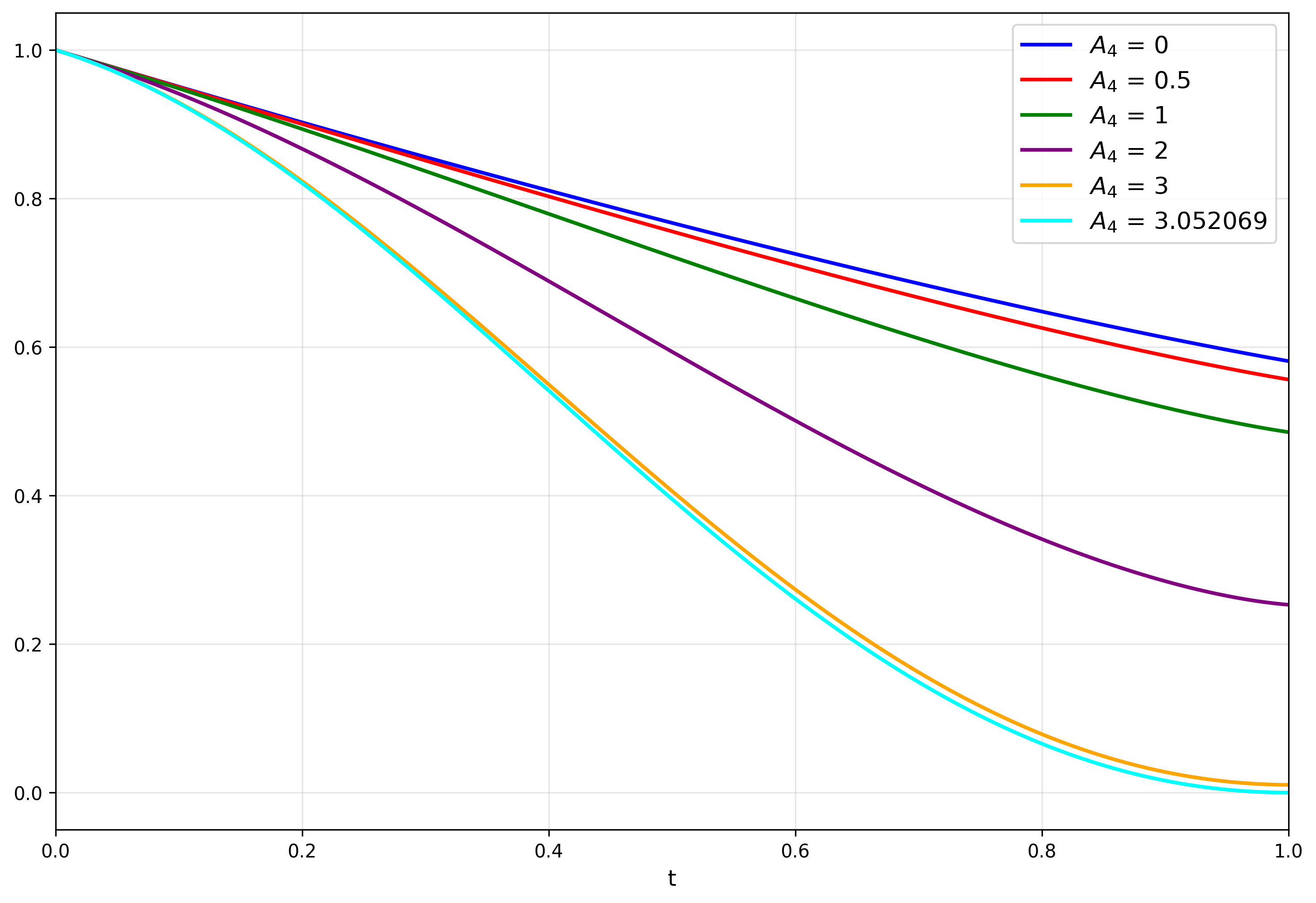}}  \\ a)
\end{minipage}
\hfill
\begin{minipage}[h]{0.5\linewidth}
\center{\includegraphics[width=1\linewidth]{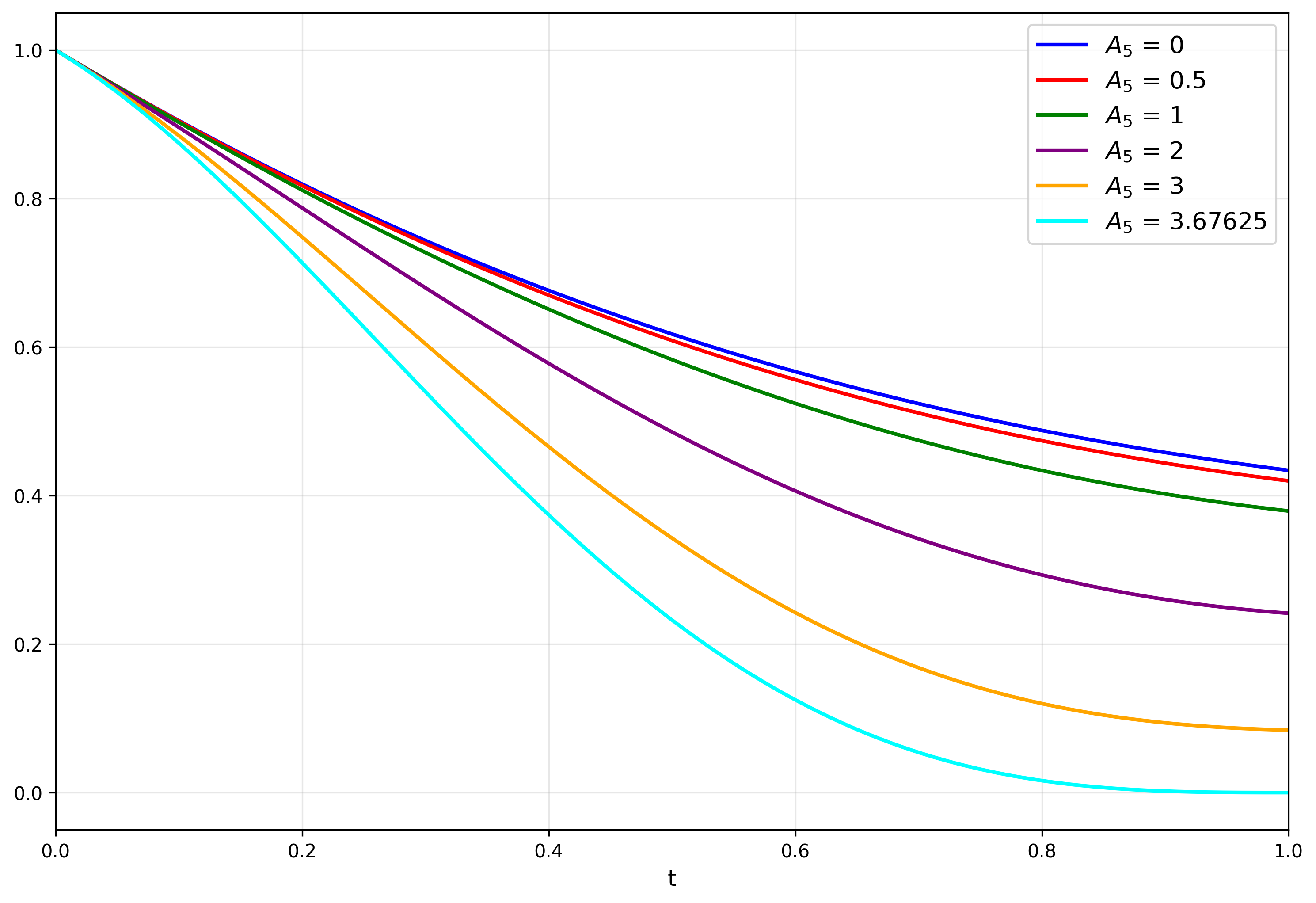}} b) \\
\end{minipage}
\vfill
\begin{minipage}[h]{0.5\linewidth}
\center{\includegraphics[width=1\linewidth]{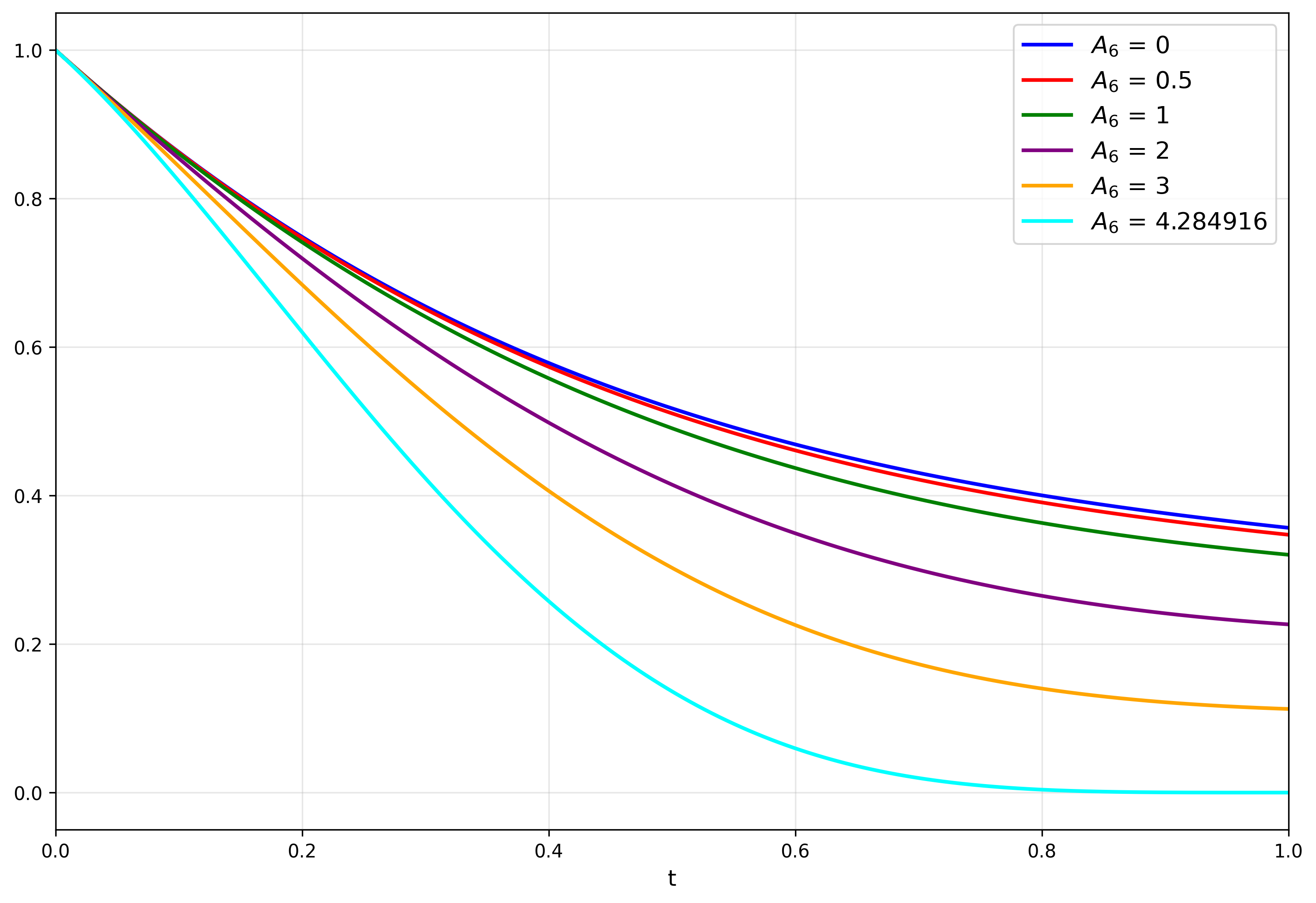}} c) \\
\end{minipage}
\hfill
\begin{minipage}[h]{0.5\linewidth}
\center{\includegraphics[width=1\linewidth]{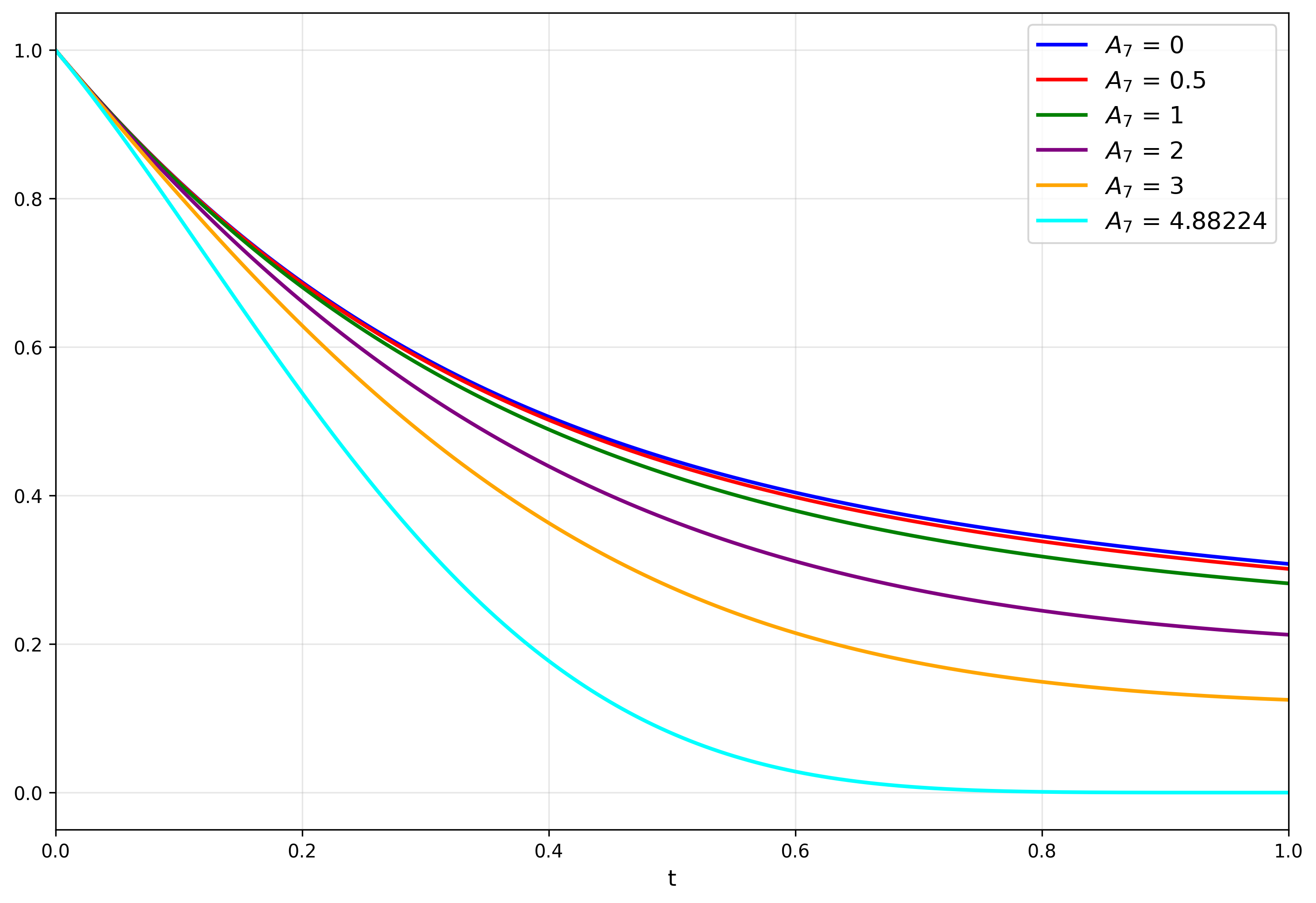}} d) \\
\end{minipage}
\caption{Graphs of $h(t,A)$: a) $n=4$, b) $n=5$, c) $n=6$, d) $n=7$.}
\label{sketch_h(n)}
\end{figure}
\bigskip
If the function is positive and has extremum at a point $t_0\in (0,1)$ then 
$$
h''(t_0) = -\left(A^2+\frac{n-3}{2t(1-t)}\right)h(t_0)<0
$$
and $t_0$ must be a maximum point. 

Above we have just proved  that $h(t,A)$ has at least one zero on (0,1). To have zeros the function must  cross the $x$-axis. Since the function $h(t,A)$ and its derivative are continuous functions then there must exist $A_0(n)$ such that $h(1,A_0(n))=0$. $A_0(n)$ is the first $A$ that give intersection of $h$ with $x$-axis.
\end{proof}

\begin{remark}
Using  Python cod we computed $h\left(1, \sqrt{j_{n/2-1}^2-1/4}\right)$ for $n=4,\ldots, 100$. We have 
$h\left(1, \sqrt{j_{n/2-1}^2-1/4}\right)<0$  for $n=4,\ldots, 100$. Consequently, there is  at least one zero in the open interval $$\left(0,  \sqrt{j_{n/2-1}^2-1/4}\right).$$
\end{remark}

\begin{corollary}\label{cor31}
The equation
$$
e^{iA}H_C\left(\frac{n-3}{2} + i A ,iA(4-n),1,3-n,2iA,1\right)=0
$$
has a unique real solution on $[0, A_0(n)]$.
\end{corollary}

\begin{remark}
We need only existence at least one zero of $h(t,A)$ in $(0,1)$. In general, there are  infinitely many zeros. See for example Fig. \ref{Fig5}. 
\end{remark}
\begin{figure}[h]
\begin{center}
\includegraphics[scale=0.42]{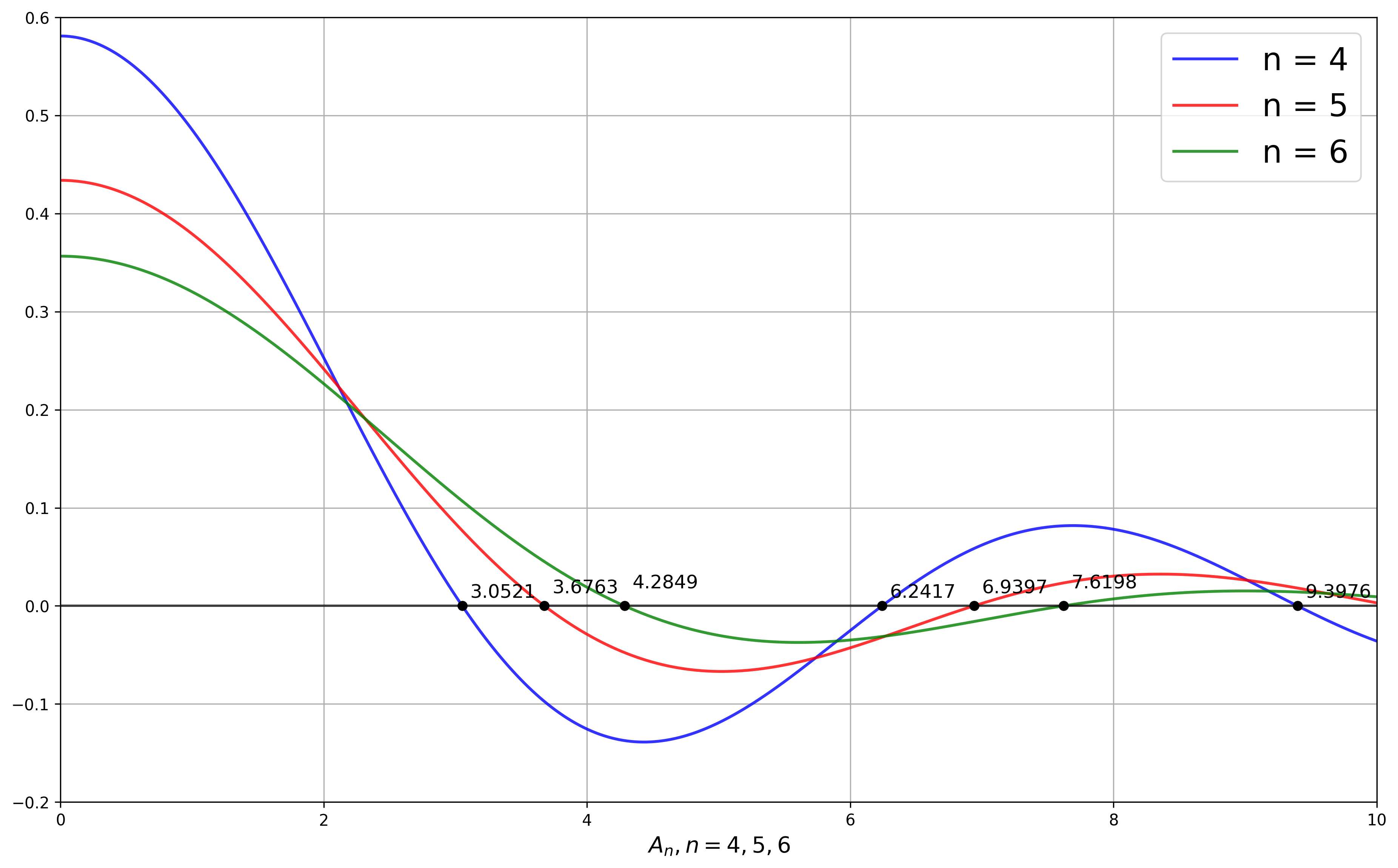}
\caption{Relations of different roots of $e^{iA}H_C\left(\frac{n-3}{2} + i A ,iA(4-n),1,3-n,2iA,1\right)$ in the case $n=4,5,6$. }
\label{Fig5}
\end{center}
\end{figure}

\begin{remark}\label{rem1}
As we remark in Section \ref{sec22} in the neighbourhood of the point $1$, the confluent Heun function decomposes into a linear combination of two independent local solutions:
$$
H_C(z) = C_1 \cdot y_1(z) + C_2 \cdot y_2(z).
$$
When we substitute $1$ in to the $H_C(z)$ then we get equation 
$$
H_C(1) = C_1 \cdot y_1(1) + C_2 \cdot y_2(1)=C_1 y_1(1) = 0
$$
that depend only on $C_1$  and we search $A$ such that $C_1=0$.
\end{remark}

\subsection{Computation of zeros}\label{sec3.4}
Solutions of the Heun Confluent equation can be contracted as series with three-term recurrence relation for the expansion  coefficients \cite{Ronveaux}. We will use power series in a neighbourhood of the regular singular point $0$ to compute $A_0(n)$.  

Suppose that 
$$
h(t, A)=\sum_{k=0}^{\infty}a_kt^{k}.
$$
As we showed above in Section \ref{sec3.2} $a_0=1$, $a_1=(3-n)/2$. For all  $k\geq 2$ we get
$$
a_{k+1} = \frac{1}{(k+1)^2} \left[ \left(k^2 - (n-3)k - \frac{n-3}{2}\right) a_k - A^2 a_{k-1} + A^2 a_{k-2} \right].
$$

To find $A_0(n)$ we must solve the equation $h(1,A) =0$. Using these  three-term recurrence relations for the expansion  coefficients, we implemented Python code (see Appendix B) that allows us to calculate 
$$
h_N(t, A)= \sum\limits_{k=0}^{N} a_k t^k
$$
for any natural $N\in \mathbb{N}$. We can easily calculate the value of $h_N(t, A)$ at $t=1$ with any given precision. Consequently, we have the equation 
$$
 h_N(1, A)= 0
$$
for any given $N$, where $ h_N(1, A)$ is a polynomial with variable $A$. 

In \cite{Ga_Nas}, we improved known lower estimates of the
Brezis-Marcus constants. Thus we have  bilateral estimates of the sharp constants $A_0(n)$ for any natural $n$ (see Table \ref{tab:my_label}). Combining  them with numerical methods like the golden section method, the bisection  method or Newton method we could compute the solutions of the   above equation. See Appendices A-D for more details. 

Note we solve polynomial equations. Since the degree is equal to $N$, we will have at most $N$ roots. It is easy to take minimum of them by standard algorithms. This is avoid us to prove that in the segment there is only one solution. We could approximate $h(t,A)$ by $h_N(t,A)$  with any given precision. Hence, an increase $N$ will not change the solutions much.

\section{Main Results}

\subsection{From multidimensional inequalities to one-dimensional ones}

The following lemma shows that to prove the conjecture, it is enough to prove one dimensional inequalities.  Since the best constant $c(n)$ does not depend on linear transformations, it suffices to consider the case $x_0 = 0$ and $\rho = 1$ (see \cite{AW_Ball} for more information). Let $B_n^0$ be the $n$-dimensional unite ball:
\begin{equation*}
B_n^0=\{x\in\mathbb{R}^n:|x|<1\}.  
\end{equation*}

\begin{lemma}\label{le4} If for some $n\geq 2$ and for any continuously differentiable  real-valued  function $f$ such that  $f(0)=0$ and  $f'^2(t)(1-t)^{n-1}\in L_1(0,1)$ the one dimensional inequality  holds
\begin{equation}\label{le_f1}
\int\limits_0^1f'^2(t)(1-t)^{n-1}dt\geq\frac{1}{4}\int\limits_0^1\frac{f^2(t)}{t^2}(1-t)^{n-1}dt+C(n)\int\limits_0^1f^2(t)(1-t)^{n-1}dt,
\end{equation}
 then  for any $g\in C_0^1(B_n^0)$ the multidimensional inequality 
\begin{equation}\label{le_f2}
    \int\limits_{B_n^0}|\nabla g(x)|^2dx \geq \frac{1}{4}\int\limits_{B_n^0}\frac{g^2(x)}{(1-|x|)^2}dx+ C(n)\int\limits_{B_n^0}g^2(x)dx
\end{equation}
is valid. Moreover, if  the constant $C(n)$ is sharp in the one dimensional inequality (\ref{le_f1}),  then it is also sharp  in the multidimensional inequality  (\ref{le_f2}).
\end{lemma}
\begin{proof}
We will use the spherical coordinates with $x = r\theta$, $|x| = r$ and the differential element $dx = r^{n-1}drd\theta$. Inequality (\ref{le_f2}) follows from the equivalent in the spherical coordinates inequality 
\begin{equation*}
\int\limits_{\mathbb{S}^{n-1}}d\theta \int\limits_0^1\left(|\nabla g(r,\theta)|^2-\frac{g^2(r,\theta)}{4(1-r)^2}-C(n)g^2(r,\theta)\right)r^{n-1}dr\geq 0.
\end{equation*}
Here by $\mathbb{S}^{n-1}$ the unite sphere is denoted. 

Taking into account that
$$
|\nabla g(r,\theta)|\geq \left|\frac{\partial g(r,\theta)}{\partial r}\right|,
$$
we get
\begin{equation*}
\int\limits_0^1\left(\left|\frac{\partial g(r,\theta)}{\partial r}\right|^2-\frac{g^2(r,\theta)}{4(1-r)^2}-C(n)g^2(r,\theta)\right)r^{n-1}dr\geq 0.
\end{equation*}

Consequently, the change $r = 1-t$ of variables gives  
\begin{equation*}
\int\limits_0^1f'^2(t)(1-t)^{n-1}dt\geq\frac{1}{4}\int\limits_0^1\frac{f^2(t)}{t^2}(1-t)^{n-1}dt+
C(n)\int\limits_0^1f^2(t)(1-t)^{n-1}dt,
\end{equation*}
where $f \in C^1[0,1]$ is defined by the equation $f(t)=g(1-t,\theta)$ for any fixed $\theta \in \mathbb{S}^{n-1}$. 
The last inequality follows from one dimensional inequality (\ref{le_f1}).

To prove that the constant $C(n)$ is sharp in the both inequalities, we use radial functions. These are functions of the form $g(x) =f(r)$, where $ f\in C^1[0,1]$ and such that $f(0)=0$. For such functions $C(n)$ being sharp in (\ref{le_f2}) is easily seen to be equivalent to the second constant $C(n)$ in (\ref{le_f1}) being sharp. 

This completes the proof of Lemma \ref{le4}.
\end{proof}

\subsection{The case $n=2$}

Our main result is the following theorem about the sharp constant $c(2) =\kappa$.

\begin{theorem} \label{th1}
Let $f$ be a continuously differentiable  real-valued  function such that  $f(0)=0$ and  $f'^2(t)(1-t) \in L_1(0,1)$. Then the inequality holds
\begin{equation*}
\int\limits_0^1f'^2(t)(1-t)dt\geq\frac{1}{4}\int\limits_0^1\frac{f^2(t)}{t^2}(1-t)dt+\kappa^2\int\limits_0^1f^2(t)(1-t)dt.
\end{equation*}
Here $\kappa=1.718347...$ is the solution of the equation 
$$
\lambda_0(0,\kappa/2)=-1/2.
$$
\end{theorem}
\begin{proof}
For a given $f$, satisfying the conditions, we consider the function $f_0$ defined by 
$$
f_0(t) = \sqrt{t} S_{0,0}^1\left(\frac{\kappa}{2},2t-1\right),
$$
where  $S_{0,0}^1$ is an  angular spheroidal function (see Fig. \ref{Fig6}) and $\kappa$ is the solution of the equation $$\lambda_0(0,\kappa/2)=-1/2.$$ 

Consider the integral
$$
A(f) = \int\limits_0^1(1-t)\left(f'(t)-\frac{f_0'(t)}{f_0(t)}f(t)\right)^2dt\geq 0.
$$
Property 2 from Section \ref{sec21} implies that the function $f_0$ has no zeros on the unite interval $(0,1)$. Thus the integral is well defined. Moreover, for sufficiently small $t$  we have the following behaviour $f_0(t) \sim \sqrt{t}$ and, as consequence, $f'_0\not\in L_2(0,1)$.

\begin{figure}[h]
\begin{center}
\includegraphics[scale=0.65]{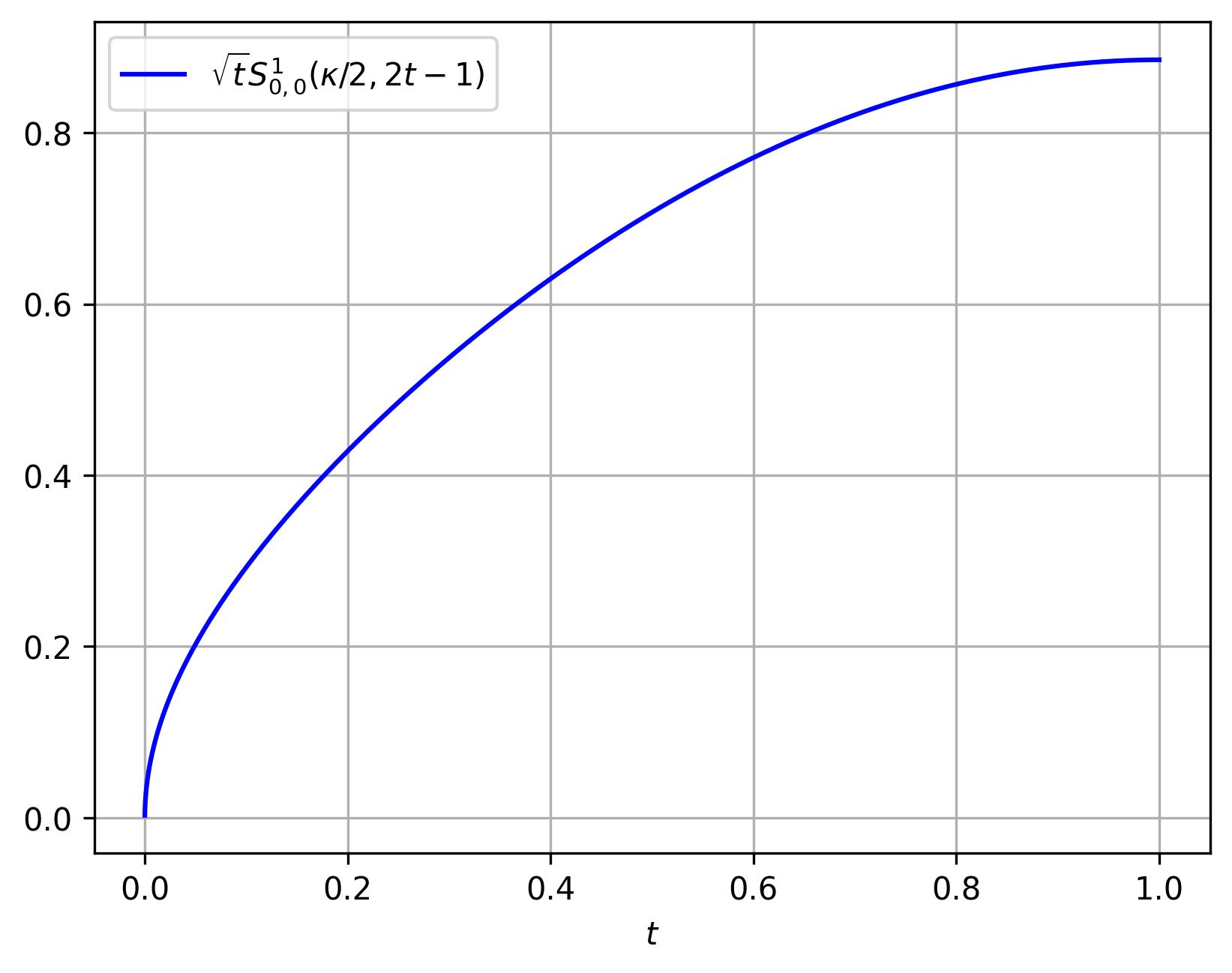}
\caption{Graphic of $\sqrt{t}S_{0,0}^1(\kappa/2,2t-1)$}
\label{Fig6}
\end{center}
\end{figure}

Straightforward computations using integration by parts give
$$
A(f) = \int\limits_0^1(1-t)f'^2(t)dt -\int\limits_0^1\frac{f_0'(t)}{f_0(t)} (1-t) df^2(t)+\int\limits_0^1\left(1-t\right)\frac{f_0'^2(t)}{f_0^2(t)}f^2(t)(1-t)dt=
$$
$$
\int\limits_0^1(1-t)f'^2(t)dt + \int\limits_0^1f^2(t)\left(\frac{f_0''(t)}{f_0(t)}-\frac{1}{1-t}\frac{f'_0(t)}{f_0(t)}\right)dt-
$$
$$
\lim_{t\to 1-}\frac{f'_0(t)}{f_0(t)} (1-t) f^2(t)+\lim_{t\to 0+}\frac{f'_0(t)}{f_0(t)} (1-t)f^2(t).
$$

The condition $f'\in L_2(0, 1)$ via the Cauchy-Schwarz inequality
$$
f^2(t) \leq \left(\int\limits_0^t|f'(\tau)|d\tau\right)^2\leq t\int\limits_0^t|f'(\tau)|^2d\tau
$$
gives $f^2(t)/t\to 0$ as $t\to 0^+$. 

Consequently, using the equality 
$$
\frac{f_0'(t)}{f_0(t)}= \frac{1}{2t}+\frac{(S_{0,0}^1\left(\frac{\kappa}{2},2t-1\right))'}{S_{0,0}^1\left(\frac{\kappa}{2},2t-1\right)},
$$
we get 
$$
\lim_{t\to 0+}\frac{f'_0(t)}{f_0(t)} (1-t)f^2(t) = 0, \quad \lim_{t\to 1-}\frac{f'_0(t)}{f_0(t)} (1-t) f^2(t)=0.
$$
Hence
$$
0<A(f) =\int\limits_0^1(1-t)f'^2(t)dt + \int\limits_0^1f^2(t)\left(\frac{f_0''(t)}{f_0(t)}-\frac{1}{1-t}\frac{f'_0(t)}{f_0(t)}\right)(1-t)dt.
$$
Applying Lemma \ref{le1}, we obtain
$$
A(f) =\int\limits_0^1(1-t)f'^2(t)dt - \int\limits_0^1f^2(t)\left(\frac{1}{4t^2}+\kappa^2\right)(1-t)dt>0.
$$
This concludes the proof Theorem \ref{th1}.
\end{proof}

In the next statement we prove that the constant $\kappa$ is sharp.

\begin{proposition}\label{prop1} For any $\varepsilon_0 >0$ there exists a function $g\in C^1[0,1]$ such that $g(0) = 0$ and $g'\in L^2[0,1]$, and 
\begin{equation*}
\int\limits_0^1g'^2(t)(1-t)dt\geq\frac{1}{4}\int\limits_0^1\frac{g^2(t)}{t^2}(1-t)dt+(\kappa^2+\varepsilon_0)\int\limits_0^1g^2(t)(1-t)dt.
\end{equation*}
\end{proposition}
\begin{proof} Consider the function  $g_\varepsilon$ defined by
$$
g_\varepsilon(t) =  t^{\frac{1+\varepsilon}{2}} S_{0,0}^1\left(\frac{\kappa}{2}, 2t-1\right) = t^{\frac{\varepsilon}{2}}f_0(t).
$$
and the following difference
\begin{equation*}
D = \frac{1}{4}\int\limits_0^1\frac{g_\varepsilon^2(t)}{t^2}(1-t)dt+(\kappa^2+\varepsilon)\int\limits_0^1g_\varepsilon^2(t)(1-t)dt - \int\limits_0^1g_\varepsilon'^2(t)(1-t)dt.
\end{equation*}
Obviously, $g_\varepsilon\in C^1[0,1]$, $g_\varepsilon(0)=0$, and $g'_\varepsilon\in L_2[0,1]$. Straightforward computations give
$$
g'_\varepsilon(t) = \frac{\varepsilon}{2}t^{\frac{\varepsilon}{2}-1}f_0(t)+t^{\frac{\varepsilon}{2}}f_0'(t).
$$
Using computations in the proof of Theorem \ref{th1}, we have
$$
D = \varepsilon_0\int\limits_0^1\frac{g_\varepsilon^2(t)}{t^2}(1-t)dt - \int\limits_0^1\left(g_\varepsilon'(t)-\frac{f_0'(t)}{f_0(t)}g_\varepsilon(t)\right)^2(1-t)dt= 
$$
$$
\varepsilon_0\int\limits_0^1g_\varepsilon^2(t)(1-t)dt - \int\limits_0^1\left(\frac{\varepsilon}{2}t^{\frac{\varepsilon}{2}-1}f_0(t)\right)^2(1-t)dt.
$$
One can show that 
$$
D = \varepsilon_0\int\limits_a^b  t^{1+\varepsilon} \left(S_{0,0}^1\left(\frac{\kappa}{2}, 2t-1\right)\right)^2  (1-t)dt -\frac{\varepsilon^2}{4} \int\limits_0^1t^{\varepsilon-1} \left(S_{0,0}^1\left(\frac{\kappa}{2}, 2t-1\right)\right)^2  (1-t)dt.
$$
Obviously, $D> 0$ for sufficiently small $\varepsilon$. This completes the proof of Proposition \ref{prop1}.
\end{proof} 

\subsection{The case $n\geq 4$}

In this case our main result is the following theorem.

\begin{theorem}\label{th2}
Let $n\in \mathbb{N}\setminus\{1,2, 3\}$, $f$ be a continuously differentiable  real-valued  function such that  $f(0)=0$ and  $f'^2(t)(1-t)^{n-1} \in L_1(0,1)$. Then the sharp inequality holds
\begin{equation*}
\int\limits_0^1f'^2(t)(1-t)^{n-1}dt\geq\frac{1}{4}\int\limits_0^1\frac{f^2(t)}{t^2}(1-t)^{n-1}dt+\kappa_n^2\int\limits_0^1f^2(t)(1-t)^{n-1}dt.
\end{equation*}
Here $\kappa_n$ is the smallest solution of the equation 
$
g_n(1) = 0,
$
where 
$$
g_n(t)=\sqrt{t}e^{i\kappa_nt}H_C\left(\frac{n-3}{2} + i \kappa_n ,i\kappa_n(4-n),1,3-n,2i\kappa_n,t\right).
$$
\end{theorem}
\begin{proof}

For a given $f$, satisfying the conditions, we consider the function $f_n$ defined by 
$$
f_n(t)=\frac{\sqrt{t}e^{iAt}}{(1-t)^{n-2}}H_C\left(\frac{n-3}{2} + i A ,iA(4-n),1,3-n,2iA,t\right)
$$
and the integral
$$
I(f,n) = \int\limits_0^1(1-t)^{n-1}\left(f'(t)-\frac{f_n'(t)}{f_n(t)}f(t)\right)^2dt.
$$
From Corollary \ref{cor31}   follows that the integral is well defined on $[0,1]$. Since $f_n(t) = \sqrt{t}(1+O(t))$ as $t\to 0$ then $f'_n\not\in L_2[0,1]$ and $I(f,n)>0$ (see Fig. \ref{sketch}). Straightforward computations using integration by parts give
$$
I(f,n) = \int\limits_0^1(1-t)^{n-1}f'^2(t)dt -\int\limits_0^1\frac{f_n'(t)}{f_n(t)} (1-t)^{n-1} df^2(t)+\int\limits_0^1\left(1-t\right)^{n-1}\frac{f_n'^2(t)}{f_n^2(t)}f^2(t)dt=
$$
$$
\int\limits_0^1(1-t)^{n-1}f'^2(t)dt + \int\limits_0^1f^2(t)\left(\frac{f_n''(t)}{f_n(t)}-\frac{n-1}{1-t}\frac{f'_n(t)}{f_n(t)}\right)(1-t)^{n-1}dt-
$$
$$
\lim_{t\to 1-}\frac{f'_n(t)}{f_n(t)} (1-t)^{n-1} f^2(t)+\lim_{t\to 0+}\frac{f'_n(t)}{f_n(t)} (1-t)^{n-1} f^2(t).
$$
Using (in addition see Fig. \ref{sketch})
$$ 
f^2(t)\to 0,\quad\text{as}\quad t\to 0, \quad\text{and}\quad \frac{f'_n(t)}{f_n(t)} \sim \frac{1}{t}  
$$
for small $t$,  we have 
$$\lim_{t\to 1-}\frac{f'_n(t)}{f_n(t)} (1-t)^{n-1} f^2(t)= 0,\quad \lim_{t\to 0+}\frac{f'_n(t)}{f_n(t)} (1-t)^{n-1} f^2(t)=0.
$$

\begin{figure}[h]
\begin{minipage}[h]{0.45\linewidth}
\center{\includegraphics[width=1\linewidth]{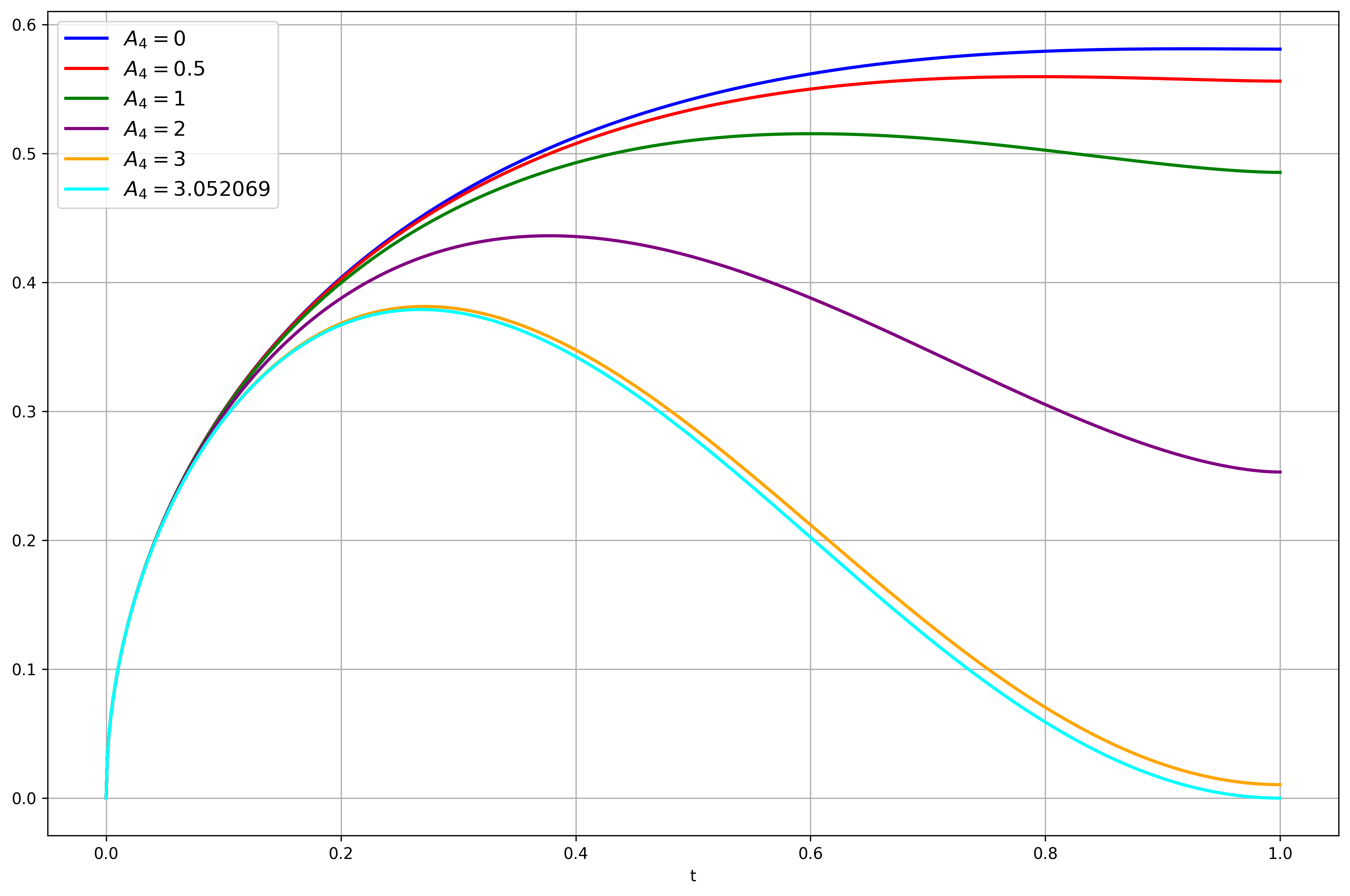}}  \\ a)
\end{minipage}
\hfill
\begin{minipage}[h]{0.45\linewidth}
\center{\includegraphics[width=1\linewidth]{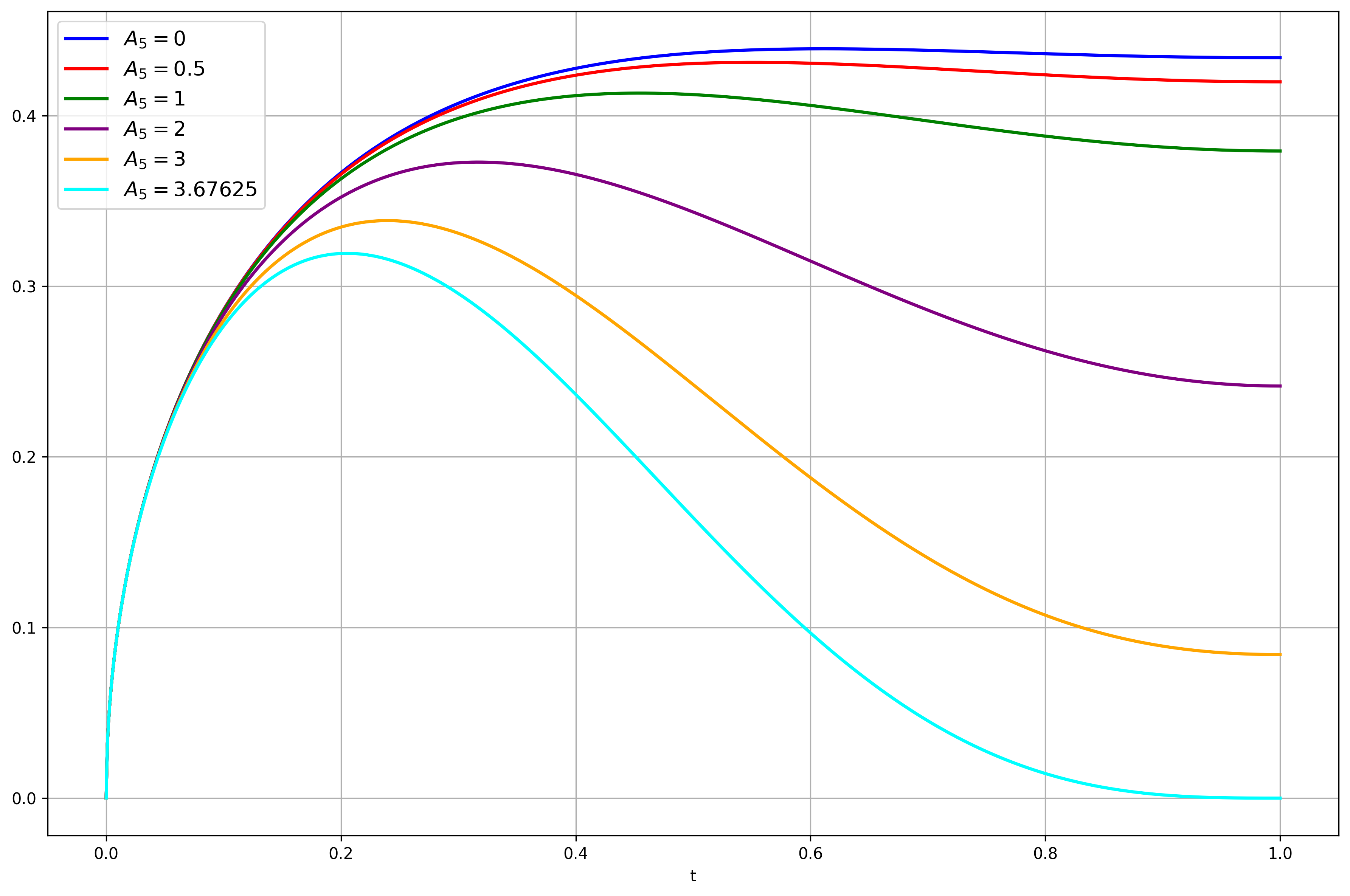}} b) \\
\end{minipage}
\vfill
\begin{minipage}[h]{0.45\linewidth}
\center{\includegraphics[width=1\linewidth]{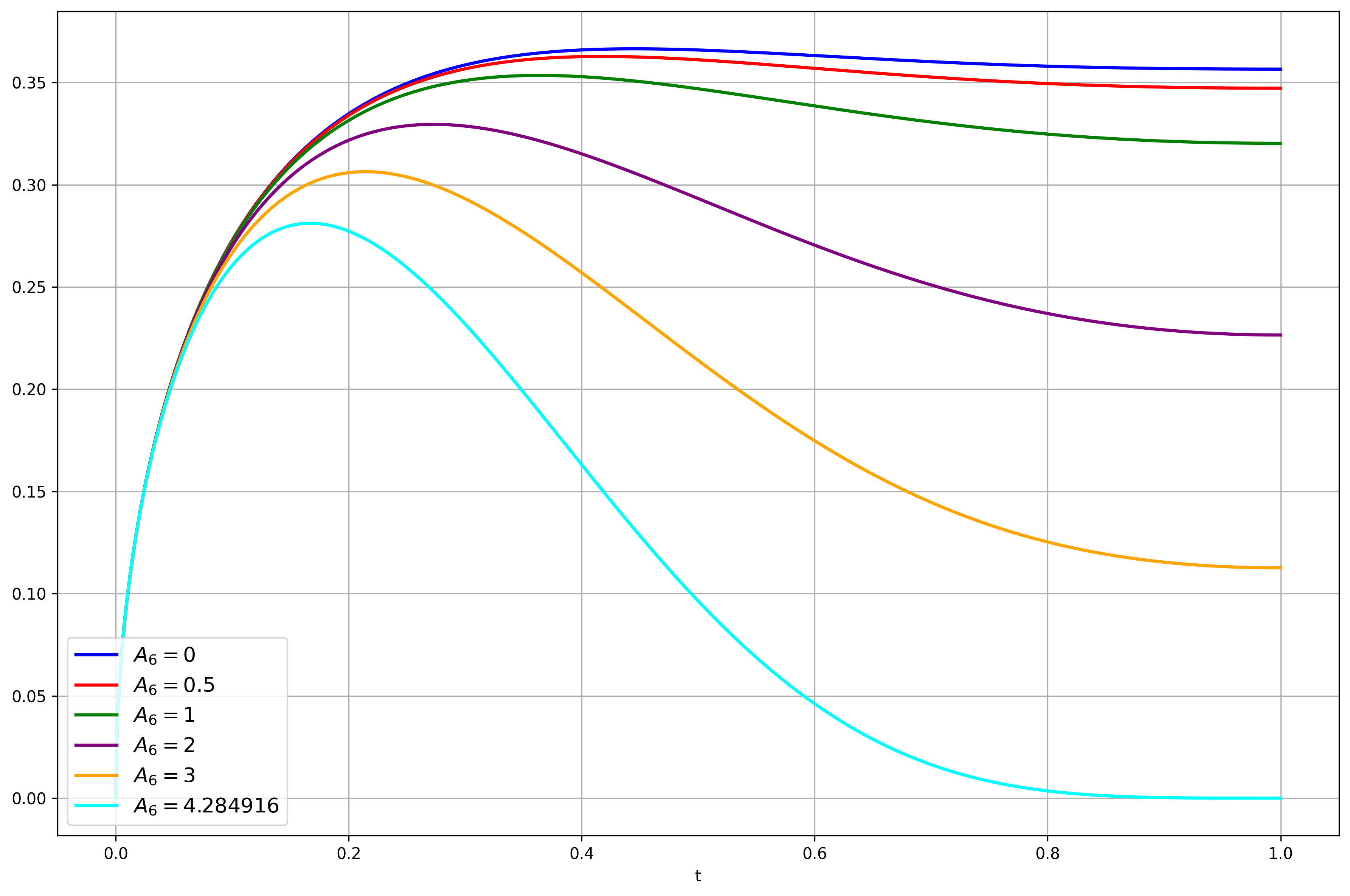}} c) \\
\end{minipage}
\hfill
\begin{minipage}[h]{0.45\linewidth}
\center{\includegraphics[width=1\linewidth]{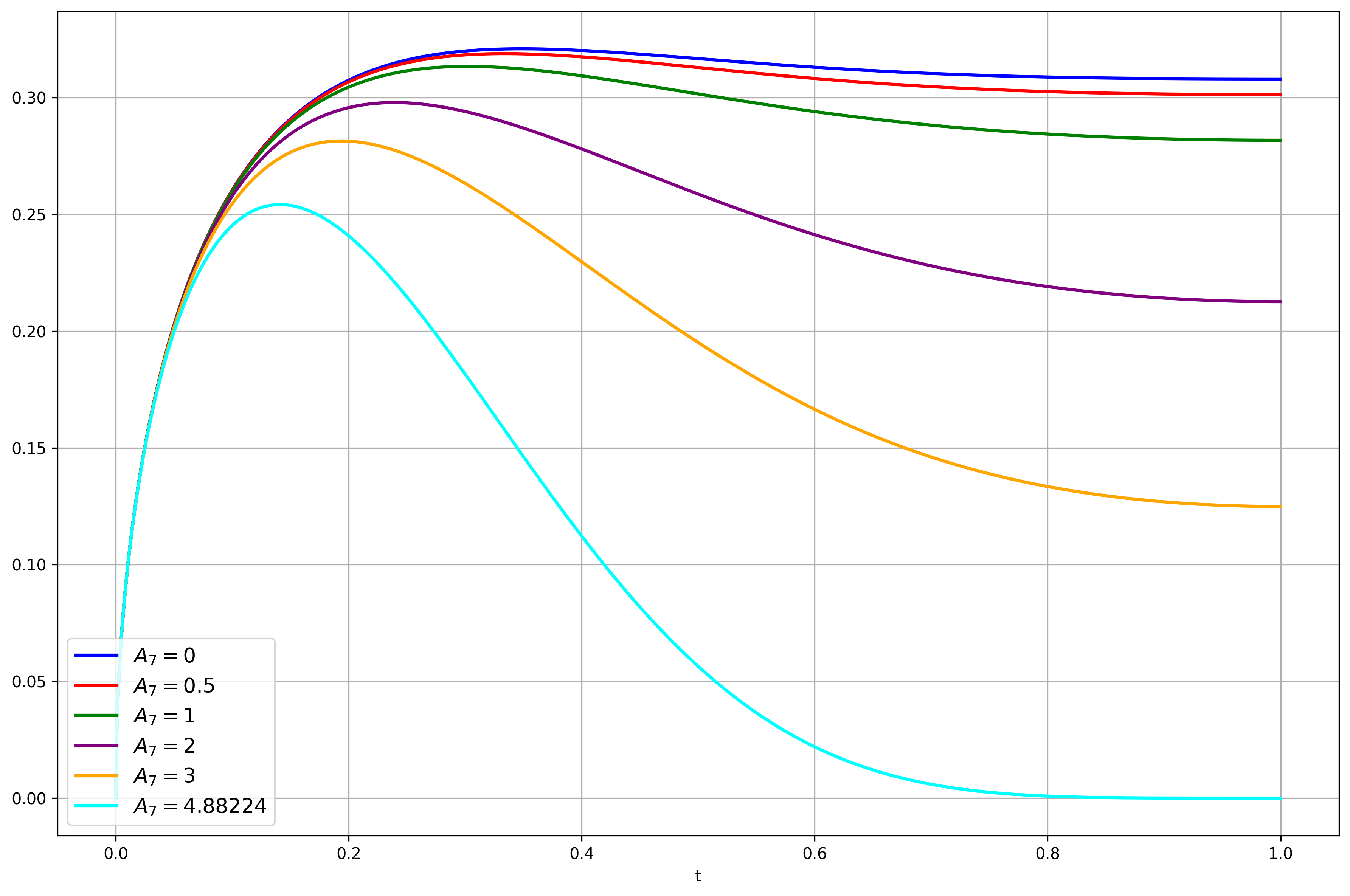}} d) \\
\end{minipage}
\caption{Graphs of $g_n(t,A)$: a) $n=4$, b) $n=5$, c) $n=6$, d) $n=7$.}
\label{sketch}
\end{figure}

Applying Lemma \ref{le2} we get 
$$
\frac{f_0''(t)}{f_0(t)}-\frac{n-1}{1-t}\frac{f'_0(t)}{f_0(t)} +\frac{1}{4t^2}+\kappa_n^2 =0.
$$
Hence
$$
\int\limits_0^1f'^2(t) (1-t)^{n-1} dt -\int\limits_0^1f^2(t)\left(\frac{1}{4t^2}+\kappa_n^2 \right)(1-t)^{n-1}dt>0.
$$
This completes the proof of Theorem \ref{th2}.
\end{proof}

In the next statement we prove that the constant $\kappa_n$ is sharp.

\begin{proposition}\label{prop2} For any $\varepsilon_0 >0$ there exists a function $g\in C^1[0,1]$ such that $g(0)  =0$ and $g'\in L^2[0,1]$, and 
\begin{equation*}
\int\limits_0^1g'^2(t)(1-t)^{n-1}dt\geq\frac{1}{4}\int\limits_0^1\frac{g^2(t)}{t^2}(1-t)dt+(\kappa_n^2+\varepsilon_0)\int\limits_0^1g^2(t)(1-t)^{n-1}dt.
\end{equation*}
\end{proposition}
\begin{proof} Consider the function  $g_\varepsilon$ defined by
$$
g_\varepsilon(t) =  \frac{t^{\frac{1+\varepsilon}{2}} e^{iAt}}{(1-t)^{n-2}}H_C\left(\frac{n-3}{2} + i A ,iA(4-n),1,3-n,2iA,t\right)= t^{\frac{\varepsilon}{2}}f_n(t)
$$
and the following difference
\begin{equation*}
I(f,n) = \frac{1}{4}\int\limits_0^1\frac{g_\varepsilon^2(t)}{t^2}(1-t)^{n-1}dt+(\kappa_n^2+\varepsilon)\int\limits_0^1g_\varepsilon^2(t)(1-t)^{n-1}dt - \int\limits_0^1g_\varepsilon'^2(t)(1-t)^{n-1}dt.
\end{equation*}
Obviously, $g_{\varepsilon}\in C^1[0,1]$, $g_{\varepsilon}(0)=0$, and $g'_\varepsilon\in L_2[0,1]$. Straightforward computations give
$$
{g'}_{\varepsilon}(t) = \frac{\varepsilon}{2}t^{\frac{\varepsilon}{2}-1}f_n(t)+t^{\frac{\varepsilon}{2}}{f'}_n(t).
$$
Computations in the proof of Theorem \ref{th2} imply 
$$
I(f,n) = \varepsilon_0\int\limits_0^1\frac{g_\varepsilon^2(t)}{t^2}(1-t)^{n-1}dt - \int\limits_0^1\left(g_\varepsilon'(t)-\frac{f_n'(t)}{f_n(t)}g_\varepsilon(t)\right)^2(1-t)^{n-1}dt= 
$$
$$
\varepsilon_0\int\limits_0^1g_\varepsilon^2(t)(1-t)^{n-1}dt - \int\limits_0^1\left(\frac{\varepsilon}{2}t^{\frac{\varepsilon}{2}-1}f_n(t)\right)^2(1-t)^{n-1}dt.
$$
One can show that 
$$
I(f,n) = \varepsilon_0\int\limits_a^b  t^{1+\varepsilon} \left(e^{2iAt}H_C^2\left(\frac{n-3}{2} + i A ,iA(4-n),1,3-n,2iA,t\right)\right)  (1-t)^{3-n}dt -
$$
$$
\frac{\varepsilon^2}{4} \int\limits_0^1t^{\varepsilon-1} \left(e^{2iAt}H_C^2\left(\frac{n-3}{2} + i A ,iA(4-n),1,3-n,2iA,t\right)\right)  (1-t)^{3-n}dt.
$$
In the neighbourhood of the point $1$, the confluent Heun function decomposes into a linear combination of two independent local solutions:
$$
H_C(z) = C_1 \cdot y_1(z) + C_2 \cdot y_2(z).
$$
We choose $A$ such that $H_C(1) = C_1 \cdot y_1(1) = 0$ (See Remark \ref{rem1}). Consequently, 
$$
\left(e^{2iAt}H_C^2\left(\frac{n-3}{2} + i A ,iA(4-n),1,3-n,2iA,t\right)\right)  (1-t)^{3-n} \sim (1-t)^{n-1} 
$$
for  $t$ closed to $1$. Therefore, $I(f,n) > 0$ for sufficiently small $\varepsilon$. 

This completes the proof of Proposition \ref{prop2}.
\end{proof} 

\section*{Acknowledgement}
The research is supported by a grant of Russian Science Foundation (project no. 23-11-00066-P,  https://rscf.ru/en/project/23-11-00066-П)

\section*{Declarations}
The authors report there are no competing interests to declare. 
\newpage
\section*{Appendix A:  Python code for sharp constants in the case $n=2$}
\addcontentsline{toc}{section}{Appendix A: Python code for sharp constants in the case $n=2$}

Using the following Python code, we compute $\kappa$ from Lemma  \ref{kappa_comp}. We used Property 7.

\bigskip
\bigskip

\begin{lstlisting}[caption= Computation of $\kappa$]
import numpy as np
from scipy.optimize import root_scalar

def continued_fraction(x, N):
    x = np.float128(x)
    term = np.float128(0.0)
    for k in range(N, 0, -1):
        a = np.float128((2*k - 1) * (2*k)) * x
        b = x - np.float128(0.5) - np.float128((2*k) * (2*k + 1))
        if k == N:
            term = a / b
        else:
            term = a / (b + term)
    return term

def equation(x, N):
    return np.float128(x - 0.5 + continued_fraction(x, N))

prev_root = None
eps = 1e-12
for N in range(1, 50):
    sol = root_scalar(
        lambda x: equation(np.float128(x), N),
        method = 'brentq',
        bracket=[np.float128(0.3), np.float128(1)])
   root = np.float128(sol.root)
    if prev_root is not None and abs(root - prev_root) < eps:
      print(f'Method converged at depth {N}, 0.5*kappa = %{np.sqrt(root)}')
      break
    prev_root = root
\end{lstlisting}

\begin{tcolorbox}[colback=white, colframe=gray!50]
Method converged at depth 8 and $0.5\kappa$ = 0.8591739836792893
\end{tcolorbox}

\newpage 
\section*{Appendix B: Python code for sharp constants in the case $n\geq 4$}
\addcontentsline{toc}{section}{Appendix B: Python code for sharp constants in the case $n\geq 4$}

Using  Python code bellow, we compute $\kappa_n$ from Theorem \ref{th2}. We used statements from Sections \ref{sec3.3} and \ref{sec3.4}.   In Python cod we use high precision ($40$ digits) to calculate constants, and Table ~\ref{tab:my_label8} from Appendix C shows the constants calculated with this parameter.  For larger $n$, we recommend increasing this parameter to improve the stability of our algorithm.

\bigskip
\bigskip

\begin{lstlisting}[caption= Computation of $\kappa_n$]
import numpy as np
from scipy.optimize import brentq
from mpmath import besseljzero
from mpmath import mp
mp.dps = 40

def h_finit_sum_mpmath(A, n, num_terms = 300):
    A_mp = mp.mpf(A)
    a = [mp.mpf(0)] * num_terms
    a[0] = mp.mpf(1)
    if num_terms > 1:
        a[1] = - mp.mpf(n - 3) / 2 * a[0]
    for k in range(1, num_terms - 1):
        term_ak = mp.mpf(k**2 - (n - 3) * k - (n - 3) / 2) * a[k]
        term_ak_minus_1 = - (A_mp**2) * a[k - 1]
        term_ak_minus_2 = (A_mp**2) * a[k - 2] if k >= 2 else mp.mpf(0)
        a[k + 1] = (term_ak + term_ak_minus_1 + term_ak_minus_2) / mp.mpf((k + 1)**2)
    return float(sum(a))

def find_A(n, num_terms = 300):
    j_nu = besseljzero(n/2-1,1)
    A2_upper = float(j_nu**2)-1/4
    if n==4:
      C_GN = 8
    else:
      beta_n = (-9+np.sqrt(12*n**2 - 48*n + 9))/(2*(n**2-4*n-6))
      C_GN = ((n**2-4*n-6)*beta_n + 6)/(4*beta_n*(1-beta_n)**2)
    a_lower = np.sqrt(C_GN)
    a_upper = np.sqrt(A2_upper)
    A_root = brentq(h_finit_sum, a_lower, a_upper, args=(n, num_terms), xtol=1e-12)
    return A_root

\end{lstlisting}

\newpage
\section*{Appendix C: Numerical values of the sharp constants}
\addcontentsline{toc}{section}{Appendix C: Numerical values of the sharp constants}
As a result of Python code  from Appendix A and  Appendix B (see also Fig. \ref{Fig2_Appen} of Appendix D) we have the results from Table \ref{tab:my_label8}.

\bigskip
\bigskip

\begin{table}[h]
    \centering
    \begin{tabular}{cccc}
     \hline
      Dimension $n$   & $\kappa_n$ &   Dimension $n$   & $\kappa_n$ 
      \\  \hline
      $2$     & $1.718347967357\ldots$ &  $32$&  $18.6853160705\ldots$ \\
      $4$     & $3.05206867992\ldots$  &  $33$&  $19.2196028463\ldots $ \\
      $5$     & $3.67625825065\ldots$  &  $34$&  $19.7532233909\ldots$ \\
      $6$     & $4.28491619153\ldots$  &  $35$&  $20.2862095607\ldots$ \\
      $7$     & $4.88224515187\ldots$  &  $36$&  $20.8185908289\ldots$\\
      $8$     & $ 5.4708663623\ldots$  &  $37$&  $21.3503945247\ldots$ \\
      $9$     & $6.05254051442\ldots$  &  $38$&  $21.8816460430\ldots$ \\
      $10$    & $6.62851541454\ldots$  &  $39$&  $22.4123690283\ldots$\\
      $11$    & $7.19971211079\ldots$  &  $40$&  $22.9425855370\ldots$\\
      $12$    & $7.76683255314\ldots$  &  $41$&  $23.4723161808\ldots$\\
      $13$    & $8.33042570624\ldots$  &  $42$&   $24.0015802541\ldots$\\
      $14$    & $8.89093013024\ldots$  &  $43$&  $24.5303958475\ldots$\\
      $15$    & $9.44870250429\ldots$  &  $44$&  $25.0587799492\ldots$\\
      $16$    & $10.0040373609\ldots$  &  $45$&  $25.5867485352\ldots$\\
      $17$    & $10.5571811252\ldots$  &  $46$&  $26.1143166509\ldots$\\
      $18$    & $11.1083423404\ldots$  &  $47$&  $26.6414984843\ldots$\\
      $19$    & $11.6576992721\ldots$  &  $48$&  $27.1683074318\ldots$\\
      $20$    & $12.2054056654\ldots$  &  $49$&  $27.6947561580\ldots$\\
      $21$    & $12.7515951734\ldots$  &  $50$&  $28.2208566497\ldots$\\
      $22$    & $13.2963848134\ldots$  &  $51$&  $28.7466202650\ldots$\\
      $23$    & $13.8398776986\ldots $ &  $52$&  $29.2720577777\ldots$\\
      $24$    & $14.3821652212\ldots $ &  $53$&  $29.7971794182\ldots$\\
      $25$    & $14.9233288167\ldots $ &  $54$&  $30.3219949103\ldots$\\
      $26$    & $15.4634414028\ldots $ &  $55$& $30.8465135056\ldots$\\
      $27$    & $16.0025685622\ldots $ &  $56$&  $31.3707440141\ldots$\\
      $28$    & $16.5407695245\ldots $&   $57$&  $31.8946948329\ldots$\\
      $29$    & $17.0780979856\ldots $&   $58$&  $32.4183739728\ldots$\\
      $30$    & $17.6146027971\ldots $&   $59$&  $32.9417890816\ldots$\\
      $31$    & $18.1503285499\ldots $&   $60$&  $33.4649474672\ldots$\\
       \hline
    \end{tabular}
    \caption{The sharp constant $\kappa_n$ for different $n= 2, 4, 5,\ldots, 60$}
    \label{tab:my_label8}
\end{table}

\newpage 
\section*{Appendix D: Graphs of $h(1,A)$ for different $n$}
\addcontentsline{toc}{section}{Appendix D: Graphs of $h(1,A)$ for different $n$}

\begin{figure}[h]
\begin{center}
\includegraphics[scale=0.6]{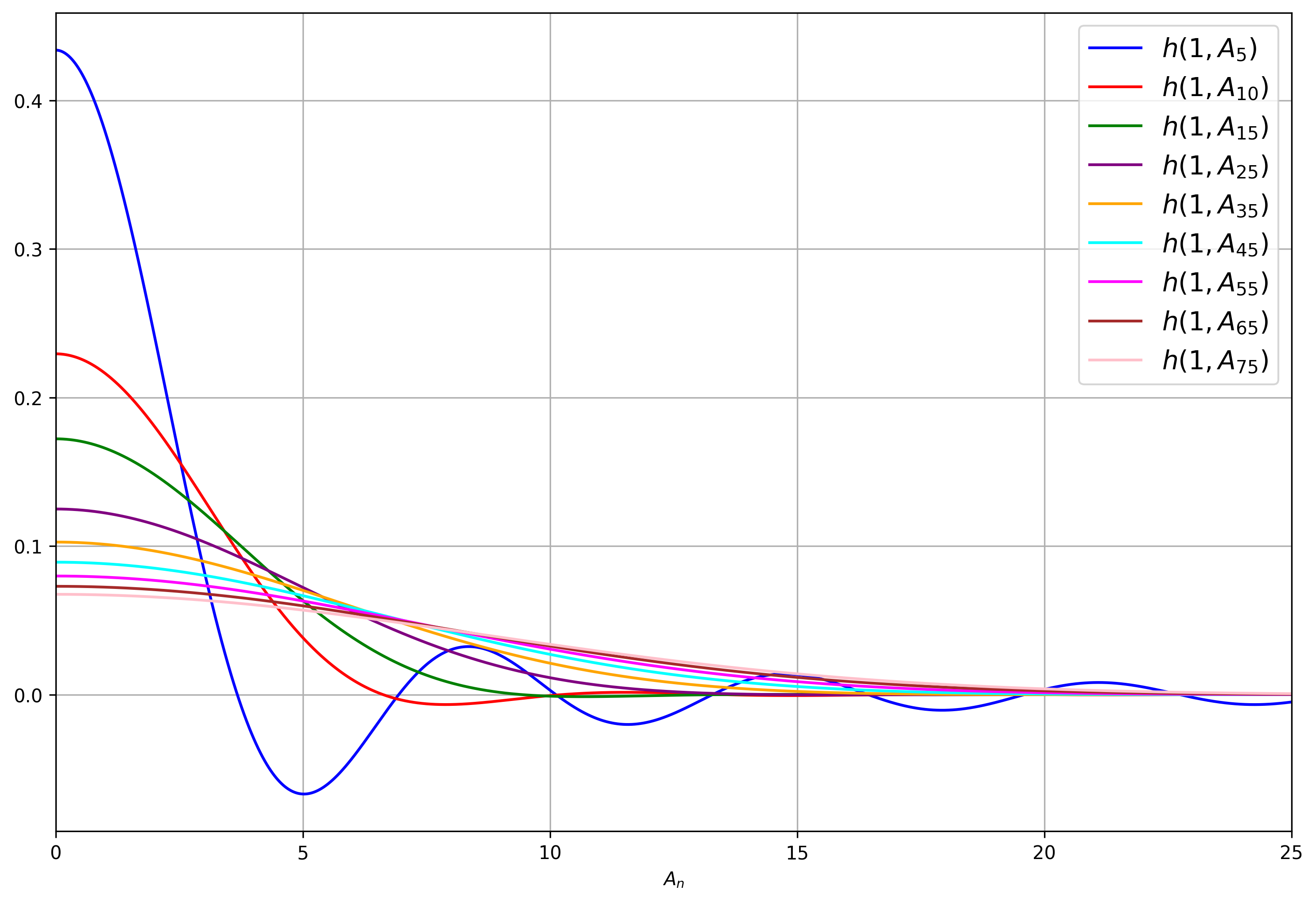}
\caption{Graphs of $h(1,A)$ for different $n$ }
\label{Fig2_Appen}
\end{center}
\end{figure}

\end{document}